\documentclass{article}
\usepackage[english]{babel}
\usepackage{tikz-cd}
\usepackage{amssymb}
\usepackage{amsthm}
\usepackage{amsmath}
\usepackage{booktabs}
\usepackage{amsfonts}
\usepackage{amsthm}
\usepackage{caption}
\usepackage{graphicx}
\usepackage{xspace}
\usepackage{hyperref}
\usepackage{geometry}
\usepackage{color}
\usepackage{url}
\usepackage{mathtools}
\usepackage[nottoc]{tocbibind}
\usetikzlibrary{decorations.pathmorphing}
\newtheorem{thm}{Theorem}[section]
\newtheorem{lem}[thm]{Lemma}
\newtheorem{prop}[thm]{Proposition}
\newtheorem{cor}[thm]{Corollary}

\theoremstyle{definition}
\newtheorem{dfn}[thm]{Definition}
\newtheorem{cns}[thm]{Construction}
\newtheorem{rmk}[thm]{Remark}

\usepackage[all]{xy}

\title{A Quillen model structure for bigroupoids and pseudofunctors}
\author{Martijn den Besten}
\begin{document}
\date{}
\maketitle

\begin{abstract}
A model structure on the category of (small) bigroupoids and pseudofunctors is constructed. In this model structure, every object is cofibrant. In order to keep certain calculations of manageable size, a coherence theorem for bigroupoids and a coherence theorem for pseudofunctors are proven, which may be of independent interest as well.
\end{abstract}

\section{Introduction}

The purpose of this paper is to construct a model structure on the category of (small) bigroupoids and pseudofunctors. In a nutshell, a model structure provides an environment in which one can do abstract homotopy theory. The notion was first introduced by Quillen in \cite{MR0223432}, but has been further refined over the years. Standard references regarding the theory of model structures are \cite{MR1650134} and \cite{MR1944041}. Some well known examples of categories carrying a model structure are the category of topological spaces, the category of simplicial sets and the category of (small) groupoids. The latter is closely related to the main category of this paper. As the name suggests, bigroupoids are a second order analog of groupoids. This analogy persists in the model structure we present below, as it highly similar to the classical model structure on the category of groupoids. The fact that the collection of 1- and 2-cells between two fixed 0-cells in a bigroupoid form a groupoid even allows us to use the model structure for groupoids to our advantage at several points in the construction. 

The model structure on bigroupoids we give here is not the first model structure on a category whose objects are 2-categorical in nature. In \cite{MR1239560}, Moerdijk and Svensson give a model structure on the category of (small) 2-groupoids and 2-functors, and in \cite{MR1931220}, Lack gives one on the category of (small) 2-categories and 2-functors. In \cite{MR2138540} Lack corrects an error made in \cite{MR1931220}, while also giving a model structure on the category of (small) bicategories and strict homomorphisms. A bicategory is a weaker variant of a 2-category, in the same way that a bigroupoid is a weaker variant of a 2-groupoid. So, we see that model structures exist both on categories with weak and categories with strict 2-categorical objects. However, a commonality of the aforementioned categories is that all their morphisms are strict.

The morphisms of the category on which we build a model structure are the pseudofunctors, which are not strict. Pseudofunctors are more general and in many aspects, they are the more natural notion of morphism to use. This is illustrated in Example 3.1 and Remark 4.4 of \cite{MR1931220}, where morphisms that `should' exist, only exist as a pseudofunctor, even if everything else is strict. It is also reflected in the fact that the cofibrations in the model structure we give below allow a more straightforward description than those of \cite{MR1239560}, \cite{MR1931220} and \cite{MR2138540}, despite using `the same' fibrations and weak equivalences. Moreover, the constructions in this paper are elementary, in the sense that no sophisticated machinery such as the small object argument or other transfinite constructions are used.

Weak morphisms are generally not as well-behaved as strict ones and can be, for this and other reasons, more difficult to work with. For example: although the category of 2-categories and 2-functors is complete and cocomplete by standard arguments, this argument breaks down if one also considers pseudofunctors. In fact, the category of 2-categories and pseudofunctors is neither complete nor cocomplete \cite{MR1931220}. A similar argument can be made for pseudofunctors in the context of bigroupoids. However, products and coproducts can be computed in the naive way, even in the presence of pseudofunctors, and in this paper we prove that certain pullbacks along pseudofunctors exist as well.

In the process of constructing our model structure, we make use of two coherence theorems, which are proven in their entirety in the appendix. The classical way to understand a coherence theorem is the following, as formulated by Mac Lane in \cite{MR1712872}:
\begin{quote}
\textit{A coherence theorem asserts: ``Every diagram commutes''; more modestly, that every diagram of a certain class commutes.}
\end{quote}
Since Mac Lane proved the first coherence theorem -- for monoidal categories in his case -- views have shifted on what is, or should be, considered a `coherence theorem' \cite{MR985657}, but for us the classical formulation remains the most useful one. At several points in the proofs below, the coherence theorems allow us to recognize that certain diagrams commute at a glance, trivializing computations that would have been very messy and laborious otherwise. The proofs of these coherence theorems draw heavily on \cite{MR723395} and \cite{MR3076451}, which are in turn based on \cite{MR641327} and \cite{MR1250465} respectively.

\section{The category of bigroupoids}

\subsection{Bigroupoids}

Before introducing bigroupoids, we will define a wider class of structures which we imaginatively name \textit{incoherent bigroupoids}. This weaker notion ignores the usual coherence conditions and is exclusively used as a convenient intermediary step in some of the constructions. Unless otherwise specified, the structures in this paper are bigroupoids.

\begin{dfn}
An \textit{incoherent bigroupoid} $\mathcal{B}$ consists of the following data:

\begin{itemize}
\item{A set $\mathcal{B}_0$ (with elements \emph{0-cells} $A, B, \ldots$)}
\item{For every combination of 0-cells $A,B$ a groupoid $\mathcal{B}(A, B)$ (with objects \emph{1-cells} $f, g, \ldots$ and arrows \emph{2-cells} $\alpha, \beta, \ldots$)}
\item{For every combination of 0-cells $A, B, C$ a functor
	\begin{align*}
	\mathbf{C}_{A, B, C} : \mathcal{B}(B, C) \times \mathcal{B}(A, B) & \longrightarrow \mathcal{B}(A, C)\\
	(g, f) & \longmapsto g * f \\
	(\beta, \alpha) & \longmapsto \beta * \alpha	
	\end{align*} }
\item{For every 0-cell $A$ a functor
	\begin{align*}
	\mathbf{U}_{A} : 1 & \longrightarrow \mathcal{B}(A, A)\\
	\bullet & \longmapsto 1_A \\
	\mathrm{id}_{\bullet} & \longmapsto \mathrm{id}_{1_A}	
	\end{align*} }
\item{For every combination of 0-cells $A, B$ a functor
	\begin{align*}
	\mathbf{I}_{A, B} : \mathcal{B}(A, B) & \longrightarrow \mathcal{B}(B, A)\\
	f & \longmapsto f^{*} \\
	\alpha & \longmapsto \alpha^{*}
	\end{align*} }
\item{For every combination of 0-cells $A, B, C, D$ a natural isomorphism
	\begin{equation*}
	\begin{tikzcd}[row sep=huge, column sep=huge]
	\mathcal{B}(C, D) \times \mathcal{B}(B, C) \times \mathcal{B}(A, B) \arrow[r, "\mathrm{id} \times \mathbf{C}_{A, B, C}"] \arrow[d, swap, "\mathbf{C}_{B, C, D} \times \mathrm{id}"] & \mathcal{B}(C, D) \times \mathcal{B}(A, C) \arrow[d, "\mathbf{C}_{A, C, D}"] \\
	\mathcal{B}(B, D) \times \mathcal{B}(A, B) \arrow[r, swap, "\mathbf{C}_{A, B, D}"]  \arrow[ru, Rightarrow, shorten >=40pt, shorten <=40pt, "\mathbf{a}_{A, B, C, D}"] & \mathcal{B}(A, D)
	\end{tikzcd}
	\end{equation*}}
\item{For every combination of 0-cells $A, B$ natural isomorphisms
	\begin{equation*}
	\begin{tikzcd}[row sep=huge, column sep=huge]
	\mathcal{B}(A, B) \times 1 \arrow[d, swap, "\mathrm{id} \times \mathbf{U}_{A}"] \arrow[dr, sloped, "\sim"{name=A}] & &[-40pt] \mathcal{B}(A, B) \arrow[r, "!"] \arrow[d, swap, "{\langle \mathbf{I}_{A,B} , \mathrm{id} \rangle}"] & 1 \arrow[d, "\mathbf{U}_{A}"] \\
	\mathcal{B}(A, B) \times \mathcal{B}(A, A) \arrow[r, swap, "\mathbf{C}_{A, A, B}"] & \mathcal{B}(A, B) & \mathcal{B}(B, A) \times \mathcal{B}(A, B) \arrow[r, swap, "\mathbf{C}_{A, B, A}"] \arrow[ru, Rightarrow, shorten >=45pt, shorten <=45pt, "\mathbf{e}_{A, B}"] & \mathcal{B}(A, A) \\[-20pt]
	1 \times \mathcal{B}(A, B) \arrow[d, swap, "\mathbf{U}_{B} \times \mathrm{id}"] \arrow[dr, sloped, "\sim"{name=B}] &  & \mathcal{B}(A, B) \arrow[r, "{\langle \mathrm{id} , \mathbf{I}_{A,B} \rangle}"] \arrow[d, swap, "!"] & \mathcal{B}(A, B) \times \mathcal{B}(B, A) \arrow[d, "\mathbf{C}_{B, A, B}"] \\
	\mathcal{B}(B, B) \times \mathcal{B}(A, B) \arrow[r, swap, "\mathbf{C}_{A, B, B}"] & \mathcal{B}(A, B) & 1 \arrow[r, swap, "\mathbf{U}_{B}"] \arrow[ru, Rightarrow, shorten >=45pt, shorten <=45pt, "\mathbf{i}_{A, B}"] & \mathcal{B}(B, B)
	\arrow[Rightarrow, shorten >=10pt, shorten <=10pt, from=2-1, to=A, "\mathbf{r}_{A,B}"]
	\arrow[Rightarrow, shorten >=10pt, shorten <=10pt, from=4-1, to=B, "\mathbf{l}_{A,B}"]
	\end{tikzcd}
	\end{equation*}}
\end{itemize}
\end{dfn}

\begin{rmk} \label{localrmk}
The properties of the groupoids $\mathcal{B}(A, B)$ are referred to as \textit{local} properties. For example, if every $\mathcal{B}(A, B)$ is discrete, it is said that $\mathcal{B}$ is locally discrete.
\end{rmk}

\begin{dfn}
A \textit{bigroupoid} $\mathcal{B}$ is an incoherent bigroupoid satisfying the following extra conditions:
\begin{itemize}
\item{For every combination
	\begin{equation*}
	A \overset{f}{\longrightarrow} B \overset{g}{\longrightarrow} C \overset{h}{\longrightarrow} D \overset{k}{\longrightarrow} E
	\end{equation*}
	of composable 1-cells, the following diagram commutes
	\begin{equation} \label{coh1}
	\begin{tikzcd}[row sep=huge, column sep=huge]
	((kh)g)f \arrow[r, "\mathbf{a} * \mathrm{id}" ] \arrow[d, swap, "\mathbf{a}" ] & (k(hg))f \arrow[r, "\mathbf{a}" ] & k((hg)f) \arrow[d, "\mathrm{id} * \mathbf{a}"] \\
	(kh)(gf) \arrow[rr, swap, "\mathbf{a}" ] & & k(h(gf))
	\end{tikzcd}
	\end{equation}}
\item{For every combination
	\begin{equation*}
	A \overset{f}{\longrightarrow} B \overset{g}{\longrightarrow} C
	\end{equation*}
	of composable 1-cells, the following diagram commutes
	\begin{equation} \label{coh2}
	\begin{tikzcd}[row sep=huge, column sep=huge]
	(g1)f \arrow[rr, "\mathbf{a}"] \arrow[dr, swap, "\mathbf{r} * \mathrm{id}"] & & g(1f) \arrow[dl, "\mathrm{id} * \mathbf{l}"] \\
	& gf &
	\end{tikzcd}
	\end{equation}}
\item{For every 1-cell
	\begin{equation*}
	A \overset{f}{\longrightarrow} B
	\end{equation*}
	the following diagram commutes
	\begin{equation} \label{coh3}
	\begin{tikzcd}[row sep=huge, column sep=huge]
	1f \arrow[r, "\mathbf{i} * \mathrm{id}" ] \arrow[d, swap, "\mathbf{l}"] & (ff^*)f \arrow[r, "\mathbf{a}"] & f(f^*f) \arrow[d, "\mathrm{id} * \mathbf{e}"] \\
	f & & f1 \arrow[ll, "\mathbf{r}"]
	\end{tikzcd}
	\end{equation}}
\end{itemize}
\end{dfn}

\begin{rmk}
We will sometimes write $-*-$ for the functor $\mathcal{C}_{A, B, C}$ and shorten $g*f$ by $gf$, for 1-cells $f$ and $g$. The action of the functor $-*-$ on 2-cells is sometimes referred to as \textit{horizontal composition}, to distinguish it from the ordinary composition of 2-cells as arrows in a category, which is in turn referred to as \textit{vertical composition} and is usually denoted by $- \circ -$.
\end{rmk}

\begin{dfn}
A \textit{strict bigroupoid} or \textit{2-groupoid} is a bigroupoid in which the natural isomorphisms $\mathbf{a}$, $\mathbf{l}$, $\mathbf{r}$, $\mathbf{e}$ and $\mathbf{i}$ are all identities.
\end{dfn}

\subsection{Morphisms of bigroupoids}

As in the previous section, we first introduce a weaker notion of morphism, which ignores coherence conditions.
\begin{dfn}
An \textit{incoherent morphism} $(F, \phi)$ from a (possibly incoherent) bigroupoid $\mathcal{B}$ to  a (possibly incoherent) bigroupoid $\mathcal{B}'$ consists of the following data:

\begin{itemize}
\item{A function 
	\begin{equation*}
	F : \mathcal{B}_{0} \longrightarrow \mathcal{B}_{0}'
	\end{equation*}}
\item{For every combination of 0-cells $A, B$ in $\mathcal{B}$ a functor 
	\begin{equation*}
	F_{A, B} : \mathcal{B}(A, B) \longrightarrow \mathcal{B}'(FA, FB)
	\end{equation*}}
\item{For every combination of 0-cells $A, B, C$ in $\mathcal{B}$ a natural isomorphism
	\begin{equation*}
	\begin{tikzcd}[row sep=huge, column sep=huge]
	\mathcal{B}(B, C) \times \mathcal{B}(A, B) \arrow[r, "\mathbf{C}_{A, B, C}"] \arrow[d, swap, "F_{B, C} \times F_{A, B}"] & \mathcal{B}(A, C) \arrow[d, "F_{A, C}"] \\
	\mathcal{B}'(FB, FC) \times \mathcal{B}'(FA, FB) \arrow[r, swap, "\mathbf{C}_{FA, FB, FC}'"] \arrow[ru, Rightarrow, shorten >=40pt, shorten <=40pt, "\phi_{A, B, C}"]  & \mathcal{B}'(FA, FC)
	\end{tikzcd}
	\end{equation*}}
\item{For every 0-cell $A$ in $\mathcal{B}$ a natural isomorphism
	\begin{equation*}
	\begin{tikzcd}[row sep=huge, column sep=huge]
	1 \arrow[r, "\mathbf{U}_{A}"] \arrow[d, swap, "\mathrm{id}"] & \mathcal{B}(A, A) \arrow[d, "F_{A, A}"] \\
	1 \arrow[r, swap, "\mathbf{U}_{FA}'"] \arrow[ru, Rightarrow, shorten >=25pt, shorten <=25pt, "\phi_{A}"] & \mathcal{B}'(FA, FA)
	\end{tikzcd}
	\end{equation*}}
\item{For every combination of 0-cells $A, B$ in $\mathcal{B}$ a natural isomorphism
	\begin{equation*}
	\begin{tikzcd}[row sep=huge, column sep=huge]
	\mathcal{B}(A, B) \arrow[r, "\mathbf{I}_{A, B}"] \arrow[d, swap, "F_{A, B}"] & \mathcal{B}(B, A) \arrow[d, "F_{B, A}"] \\
	\mathcal{B}'(FA, FB) \arrow[r, swap, "\mathbf{I}_{FA, FB}'"] \arrow[ru, Rightarrow, shorten >=30pt, shorten <=30pt, "\phi_{A, B}"] & \mathcal{B}'(FB, FA)
	\end{tikzcd}
	\end{equation*}}
\end{itemize}
\end{dfn}

\begin{rmk}
The properties of the functors $F_{A, B}$ are referred to as \textit{local} properties. For example, if every $F_{A, B}$ is faithful, it is said that $(F, \phi)$ is locally faithful. (This is similar to Remark \ref{localrmk}.)
\end{rmk}

\begin{dfn}
A \textit{morphism} $(F, \phi)$ from a (possibly incoherent) bigroupoid $\mathcal{B}$ to  a (possibly incoherent) bigroupoid $\mathcal{B}'$ is an incoherent morphism satisfying the following extra conditions:
\begin{itemize}
\item{For every combination
	\begin{equation*}
	A \overset{f}{\longrightarrow} B \overset{g}{\longrightarrow} C \overset{h}{\longrightarrow} D 
	\end{equation*}
	of composable 1-cells, the following diagram commutes
\begin{equation}\label{coh4}
	\begin{tikzcd}[row sep=huge, column sep=huge]
	(Fh * Fg) * Ff \arrow[r, "\phi * \mathrm{id}"] \arrow[d, swap, "\mathbf{a}'"] & F(h * g) * Ff \arrow[r, "\phi"] & F((h * g) * f) \arrow[d, "F\mathbf{a}"] \\
	Fh * (Fg * Ff) \arrow[r, swap, "\mathrm{id} * \phi"] & Fh * F(g * f) \arrow[r, swap, "\phi"] & F(h * (g * f))
	\end{tikzcd}
	\end{equation}}
\item{For every 1-cell
	\begin{equation*}
	A \overset{f}{\longrightarrow} B
	\end{equation*}
	the following diagrams commute
	\begin{equation}\label{coh5}
	\begin{tikzcd}[row sep=huge]
	Ff * 1_{FA} \arrow[r, "\mathrm{id} * \phi"] \arrow[d, swap, "\mathbf{r}'"] & Ff * F 1_{A} \arrow[r, "\phi"] & F(f * 1_{A}) \arrow[d, "F \mathbf{r}"] &[-20pt] 1_{FB} * Ff \arrow[r, "\phi * \mathrm{id}"] \arrow[d, swap, "\mathbf{l}'"] & F 1_{B} * Ff \arrow[r, "\phi"] & F(1_{B} * f) \arrow[d, "F \mathbf{l}"] \\
	Ff \arrow[rr, swap, "\mathrm{id}"] & & Ff & Ff \arrow[rr, swap, "\mathrm{id}"] & & Ff \\[-20pt]
	(Ff)^{*} * Ff \arrow[r, "\phi * \mathrm{id}" ] \arrow[d, swap, "\mathbf{e}'"] & F(f^{*}) * Ff \arrow[r, "\phi"] & F( f^{*} * f) \arrow[d, "F \mathbf{e}"] & 1_{FB} \arrow[d, swap, "\mathbf{i}'" ] \arrow[rr, "\phi"] & & F 1_{B} \arrow[d, "F \mathbf{i}"] \\
	1_{FA} \arrow[rr, swap, "\phi"] & & F 1_{A} & Ff * (Ff)^{*} \arrow[r, swap, "\mathrm{id} * \phi"] & Ff * F (f^{*}) \arrow[r, swap, "\phi"] & F(f * f^{*})
	\end{tikzcd}
	\end{equation}}
\end{itemize}
\end{dfn}

\begin{rmk}
These types of morphisms are sometimes referred to as \textit{pseudofunctors} or \textit{weak 2-functors}, since they are not, in general, structure preserving maps. A morphism $(F, \phi)$ for which $\phi = \mathrm{id}$ and which therefore does preserves all structure (not just up to isomorphism) is called \textit{strict}.
\end{rmk}

The composition of two (possibly incoherent) morphisms $(F, \phi) : \mathcal{B} \longrightarrow \mathcal{B}'$ and $(G, \psi) : \mathcal{B}' \longrightarrow \mathcal{B}''$ is given by
\begin{equation*}
(G, \psi) \circ (F, \phi) = (G \circ F, G \phi \circ \psi F) : \mathcal{B} \longrightarrow \mathcal{B}''
\end{equation*}
Here, $G \phi \circ \psi F$ represents the pasting of diagrams, as in:
\begin{equation*}
\begin{tikzcd}[row sep=huge, column sep=huge]
\mathcal{B}(A, B) \arrow[r, "\mathbf{I}_{A, B}"] \arrow[d, swap, "F_{A, B}"] & \mathcal{B}(B, A) \arrow[d, "F_{B, A}"] \\
\mathcal{B}'(FA, FB) \arrow[r, "\mathbf{I}_{FA, FB}"] \arrow[d, swap, "G_{FA, FB}"] \arrow[ru, shorten >=35pt, shorten <=35pt, Rightarrow, "\phi_{A, B}"] & \mathcal{B}'(FB, FA) \arrow[d, "G_{FB, FA}"] \\
\mathcal{B}''(GFA, GFB) \arrow[r, swap, "\mathbf{I}_{GFA, GFB}"] \arrow[ru, shorten >=35pt, shorten <=35pt, Rightarrow, "\psi_{FA, FB}"] & \mathcal{B}''(GFB, GFA) 
\end{tikzcd}
\end{equation*}
This operation is clearly associative with identity.

\begin{rmk}
In many of the upcoming proofs, we need to make separate constructions concerning composition, inversion and identity respectively. However, since these three types of constructions are usually highly similar, we will generally only provide the one for composition. We will not mention this omission in every individual proof.
\end{rmk}

Let us prove two useful lemmas which show that maps and structures can `inherit' coherence properties to some extent.

\begin{lem}\label{lem1}
Let
\begin{equation*}
\begin{tikzcd}[row sep=huge, column sep=huge]
\mathcal{A} \arrow[r, "{(F, \phi)}"] \arrow[dr, swap, "{(H, \eta)}"] & \mathcal{B} \arrow[d, "{(G, \gamma)}"] \\
& \mathcal{C}
\end{tikzcd}
\end{equation*}
be a commutative diagram of incoherent morphisms between (possibly incoherent) bigroupoids. If two of the following conditions are satisfied, then so is the third:
\begin{description}
\item[(1)] The diagrams (\ref{coh4}) and (\ref{coh5}) commute for $\gamma F$.
\item[(2)] The diagrams (\ref{coh4}) and (\ref{coh5}) commute for $\phi$, after $G$ is applied to them.
\item[(3)] The diagrams (\ref{coh4}) and (\ref{coh5}) commute for $\eta$.
\end{description}
\end{lem}

\begin{proof}
We only consider $\mathbf{a}$. The proofs for $\mathbf{l}$, $\mathbf{r}$, $\mathbf{e}$ and $\mathbf{i}$ are similar. The commutativity of the left inner rectangle, the right inner rectangle and the perimeter of the following diagram correspond to condition \textbf{(1)}, \textbf{(2)} and \textbf{(3)}, respectively.
\begin{equation*}
\begin{tikzcd}[row sep=huge, column sep=huge]
& & \cdot \arrow[dr, swap, "\gamma F"] \arrow[drr, "\eta"] & & \\
\cdot \arrow[r, swap, "\gamma F * \mathrm{id}"] \arrow[rru, "\eta * \mathrm{id}"] \arrow[d, swap, "\mathbf{a}"] & \cdot \arrow[r, swap, "\gamma F"] \arrow[ru, swap, "G \phi * \mathrm{id}"] & \cdot \arrow[r, swap, "G( \phi * \mathrm{id})"] \arrow[d, "G \mathbf{a}"] & \cdot \arrow[r, swap, "G \phi"] & \cdot \arrow[d, "GF \mathbf{a}"] \\
\cdot \arrow[r, "\mathrm{id} * \gamma F"] \arrow[rrd, swap, "\mathrm{id} * \eta"] & \cdot \arrow[r, "\gamma F"] \arrow[rd, "\mathrm{id} * G \phi"] & \cdot \arrow[r, "G( \mathrm{id} * \phi)"] & \cdot \arrow[r, "G \phi"] & \cdot \\
& & \cdot \arrow[ru, "\gamma F"] \arrow[rru, swap, "\eta"] & & \\
\end{tikzcd}
\end{equation*}
Since the other components of the diagram commute by naturality of $\gamma$ and the fact that $(G, \gamma) \circ ( F, \phi) = (H, \eta)$, irrespective of the three conditions, this proves the lemma.
\end{proof}

\begin{cor}
Morphisms between bigroupoids are closed under composition, so the collection of bigroupoids forms a category.
\end{cor}

\begin{proof}
This follows directly from $\textbf{(1)} + \textbf{(2)} \Longrightarrow \textbf{(3)}$ of Lemma \ref{lem1}.
\end{proof}

\begin{lem}\label{lem2}
Let $(F, \phi) : \mathcal{A} \longrightarrow \mathcal{B}$ be a morphism between incoherent bigroupoids. Then the following are equivalent:
\begin{description}
\item[(1)] The diagrams (\ref{coh1}), (\ref{coh2}) and (\ref{coh3}) commute for 1-cells in the image of $F$.
\item[(2)] The diagrams (\ref{coh1}), (\ref{coh2}) and (\ref{coh3}) commute, after $F$ is applied to them.
\end{description}
\end{lem}

\begin{proof}
We only consider (\ref{coh2}). The proofs for (\ref{coh1}) and (\ref{coh3}) are similar. The commutativity of the innermost triangle and outermost triangle of following diagram correspond to condition \textbf{(1)} and \textbf{(2)}, respectively.
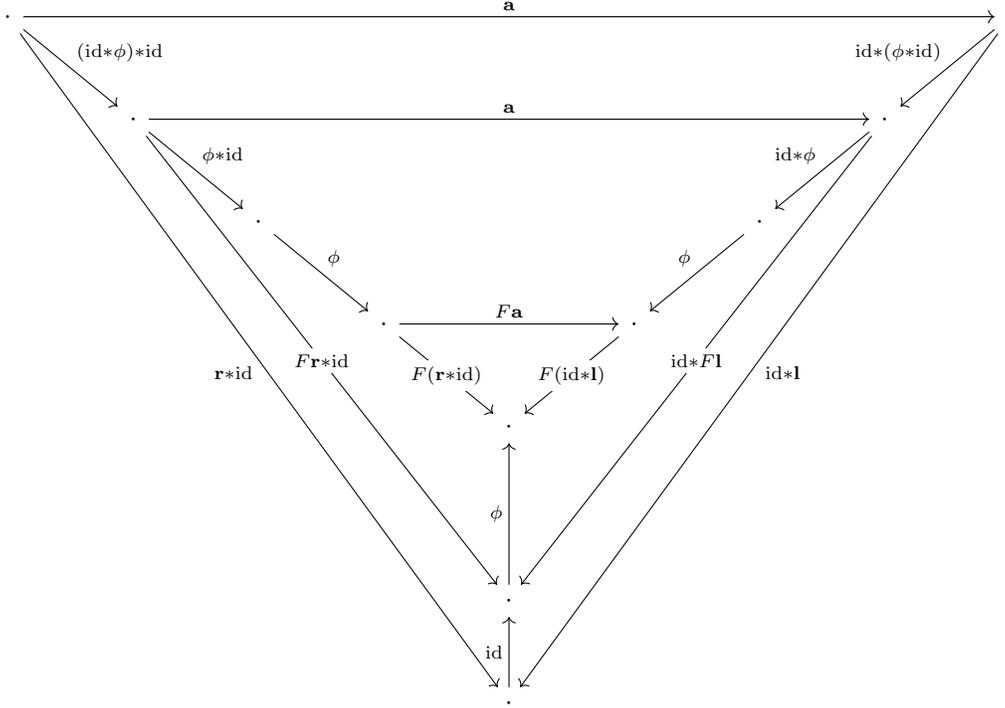
\begin{equation*}
\begin{tikzcd}[row sep=large, column sep=large]
\cdot \arrow[rrrrrrrr, "\mathbf{a}"] \arrow[dr, "(\mathrm{id} * \phi) *\mathrm{id}"] \arrow[dddddddrrrr, swap, "\mathbf{r} * \mathrm{id}"] & & & & & & & & \cdot \arrow[dl, swap, "\mathrm{id} * ( \phi * \mathrm{id} )"] \arrow[dddddddllll, "\mathrm{id} * \mathbf{l}"] \\
& \cdot \arrow[rrrrrr, "\mathbf{a}"] \arrow[dr, "\phi * \mathrm{id}"] \arrow[dddddrrr, "F\mathbf{r} * \mathrm{id}" description] & & & & & & \cdot \arrow[dl, swap, "\mathrm{id} * \phi"] \arrow[dddddlll, swap, "\mathrm{id} * F \mathbf{l}" description] & \\
& & \cdot \arrow[dr, "\phi"] & & & & \cdot \arrow[dl, swap, "\phi"] & & \\
& & & \cdot \arrow[rr, "F \mathbf{a}"] \arrow[dr, swap, "F( \mathbf{r} * \mathrm{id})" description] & & \cdot \arrow[dl, "F( \mathrm{id} * \mathbf{l} )" description] & & & \\
& & & & \cdot & & & & \\
& & & & & & & & \\
& & & & \cdot \arrow[uu, "\phi"] & & & & \\
& & & & \cdot \arrow[u, "\mathrm{id}"] & & & &
\end{tikzcd}
\end{equation*}
Since the other components of the diagram commute by naturality of $\phi$ and the fact that $( F, \phi)$ is a morphism, irrespective of the two conditions, this proves the lemma.
\end{proof}

\section{Model structures}

Since there exist multiple nonequivalent definitions in the literature of what constitutes a model structure, we give a brief description of what we consider to be a model structure here.

\begin{dfn}
Let $f$ and $g$ be morphisms in a category $\mathcal{C}$. If for every commutative square
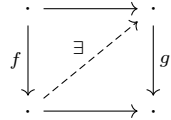
\begin{equation*}
\begin{tikzcd} [row sep=large, column sep=large]
\cdot \arrow[r] \arrow[d, swap, "f"] &  \cdot \arrow[d, "g"] \\
\cdot \arrow[ru, dashed, "\exists"] \arrow[r] &    \cdot 
\end{tikzcd}
\end{equation*}
a diagonal arrow exists as indicated in the diagram, then we say that \textit{$f$ has the left lifting property with respect to $g$} or, equivalently, that \textit{$g$ has the right lifting property with respect to $f$}.
\end{dfn}

\begin{dfn}
A \textit{weak factorization system} on a category $\mathcal{C}$ is a pair $( \mathcal{L}, \mathcal{R} )$ of classes of morphisms in $\mathcal{C}$ such that
\begin{description}
\item[(1)] any morphism in $\mathcal{C}$ can be factored as a morphism of $\mathcal{L}$ followed by a morphism of $\mathcal{R}$, and
\item[(2)] $\mathcal{L}$ consists precisely of those morphisms having the left lifting property with respect to every morphism in $\mathcal{R}$, and symmetrically, $\mathcal{R}$ consists precisely of those morphisms having the right lifting property with respect to every morphism in $\mathcal{L}$.
\end{description}
\end{dfn}

\begin{dfn}
A \textit{model structure} on a category $\mathcal{M}$ consists of three classes $\mathcal{F}$, $\mathcal{C}$ and $\mathcal{W}$ of morphisms in $\mathcal{M}$, called \emph{fibrations}, \emph{cofibrations} and \emph{weak equivalences} respectively, such that
\begin{description}
\item[(1)] $\mathcal{W}$ contains all isomorphisms and is closed under $2$-out-of-$3$, meaning that whenever the composition $g \circ f$ is defined and two of $f$, $g$ and $g \circ f$ lie in $\mathcal{W}$, then so does the third, and
\item[(2)] both $( \mathcal{C}, \mathcal{F} \cap \mathcal{W})$ and $( \mathcal{C} \cap \mathcal{W}, \mathcal{F})$ are weak factorization systems on $\mathcal{M}$.
\end{description}
\end{dfn}

\begin{rmk}
The classes $\mathcal{F} \cap \mathcal{W}$ and $\mathcal{C} \cap \mathcal{W}$ are commonly called the \textit{trivial fibrations} and \textit{trivial cofibrations} respectively.
\end{rmk}

We can now formulate the main theorem of this paper.

\begin{thm} \label{mainthm}
The category of bigroupoids and pseudofunctors carries a model structure, with fibrations, cofibrations and weak equivalences as given in Definitions \ref{dfn1}, \ref{dfn2} and \ref{dfn3} below.
\end{thm}

\begin{dfn}\label{dfn1}
A morphism $F : \mathcal{A} \longrightarrow \mathcal{B}$ is said to be a \emph{fibration} if it satisfies the following two conditions:
\begin{description}
\item[(1)] For every 0-cell $A'$ in $\mathcal{A}$ and every 1-cell $b : B \longrightarrow FA'$ in $\mathcal{B}$ there exists a 1-cell $a : A \longrightarrow A'$ in $\mathcal{A}$ such that $FA = B$ and $Fa = b$.
\item[(2)] For every 1-cell $a' : A \longrightarrow A'$ in $\mathcal{A}$ and every 2-cell $\beta : b \longrightarrow Fa'$ there exists a 2-cell $\alpha : a \longrightarrow a'$ in $\mathcal{A}$ such that $Fa = b$ and $F \alpha = \beta$.
\end{description}
\end{dfn}

\begin{dfn}\label{dfn2}
A morphism $F : \mathcal{A} \longrightarrow \mathcal{B}$ is said to be a \emph{cofibration} if it satisfies the following two conditions:
\begin{description}
\item[(1)] The function $F : \mathcal{A}_0 \longrightarrow \mathcal{B}_0$ is injective.
\item[(2)] For every combination of 0-cells $A, A'$ in $\mathcal{A}$, the functor $F_{A, A'} : \mathcal{A}( A, A') \longrightarrow \mathcal{B}( FA, FA')$ is injective on objects.
\end{description}
\end{dfn}

\begin{dfn}\label{dfn3}
A morphism $F : \mathcal{A} \longrightarrow \mathcal{B}$ is said to be a \emph{weak equivalence} if it satisfies the following two conditions:
\begin{description}
\item[(1)] For every 0-cell $B$ in $\mathcal{B}$ there exists a 0-cell $A'$ in $\mathcal{A}$ and a 1-cell $b : B \longrightarrow FA'$ in $\mathcal{B}$.
\item[(2)] For every combination of 0-cells $A, A'$ in $\mathcal{A}$, the functor $F_{A, A'} : \mathcal{A}( A, A') \longrightarrow \mathcal{B}( FA, FA')$ is an equivalence of categories.
\end{description}
\end{dfn}

\begin{rmk}
A morphism satisfying the conditions of Definition \ref{dfn3} is also known as a \textit{biequivalence}. Notice that when a morphism $F : \mathcal{A} \longrightarrow \mathcal{B}$ is in class $\mathcal{X}$ (fibrations, cofibrations, or weak equivalences), then $F$ is locally in class $\mathcal{X}$ of the canonical model structure on the category of groupoids. This is precisely the second part of Definitions \ref{dfn1}, \ref{dfn2} and \ref{dfn3}. Also note that the trivial fibrations may be characterized as those weak equivalences that are surjective on 0-cells and locally surjective on objects (1-cells).
\end{rmk}

\begin{lem} \label{lem8}
\hfill
\begin{description}
\item[(1)] Every isomorphism is a weak equivalence.
\item[(2)] The weak equivalences satisfy the \textit{2-out-of-3} property.
\item[(3)] The fibrations, cofibrations and weak equivalences are closed under retracts.
\end{description}
\end{lem}

\begin{proof}
Straightforward.
\end{proof}

\section{The cofibration - trivial fibration WFS}

In this section, we aim to prove the following proposition.

\begin{prop}
The cofibrations and trivial fibrations form a weak factorization system.
\end{prop}

By the retract argument, it suffices to show that the cofibrations have the left lifting property with respect to the trivial fibrations and that every morphism factors as a cofibration followed by a trivial fibration.

\subsection{Lifting property}

\begin{lem}
The cofibrations have the left lifting property with respect to the trivial fibrations.
\end{lem}

\begin{proof}
Given a commutative square
\begin{equation} \label{cotrflift}
\begin{tikzcd}[row sep=huge, column sep=huge]
\mathcal{A} \arrow[r, "{(F, \phi)}"] \arrow[d, swap, "{(K, \kappa)}"] & \mathcal{B} \arrow[d, "{(G, \gamma)}"] \\
\mathcal{D} \arrow[r, swap, "{(H, \eta)}"] \arrow[ru, dashed, "{\exists (L, \lambda)}"] & \mathcal{C}
\end{tikzcd}
\end{equation}
in which $K$ is a cofibration and $G$ is a trivial fibration, we construct a diagonal filler $L$, as indicated in the diagram.

Let $L : \mathcal{D}_{0} \longrightarrow \mathcal{B}_{0}$ be a function which makes the diagram
\begin{equation*}
\begin{tikzcd}[row sep=huge, column sep=huge]
\mathcal{A}_{0} \arrow[r, "F"] \arrow[d, swap, "K"] & \mathcal{B}_{0} \arrow[d, "G"] \\
\mathcal{D}_{0} \arrow[r, swap, "H"] \arrow[ru, dashed, "\exists L"] & \mathcal{C}_{0}
\end{tikzcd}
\end{equation*}
commute. Such a function exists because $K : \mathcal{A}_{0} \longrightarrow \mathcal{D}_{0}$ is injective and $G : \mathcal{B}_{0} \longrightarrow \mathcal{C}_{0}$ is surjective.

Given a pair of 0-cells $D$, $D'$ both in the image of $K$, say $D = KA$ and $D' = KA'$, we define $L_{D, D'} : \mathcal{D}(D, D') \longrightarrow \mathcal{B}(LD, LD')$ by taking a diagonal
\begin{equation*}
\begin{tikzcd}[row sep=huge, column sep=huge]
\mathcal{A}(A, A') \arrow[r, "F_{A, A'}"] \arrow[d, swap, "K_{A, A'}"] & \mathcal{B}(LD, LD') \arrow[d, "G_{LD, LD'}"] \\
\mathcal{D}(D, D') \arrow[r, swap, "H_{D, D'}"] \arrow[ru, dashed, "\exists L_{D, D'}"] & \mathcal{C}(HD, HD')
\end{tikzcd}
\end{equation*}
which exists by the model structure on the category of groupoids. Given a pair of 0-cells $D, D'$ not both in the image of $K$, we define $L_{D, D'} : \mathcal{D}(D, D') \longrightarrow \mathcal{B}(LD, LD')$ by taking a diagonal
\begin{equation*}
\begin{tikzcd}[row sep=huge, column sep=huge]
0 \arrow[r, "!"] \arrow[d, swap, "!"] & \mathcal{B}(LD, LD') \arrow[d, "G_{LD, LD'}"] \\
\mathcal{D}(D, D') \arrow[r, swap, "H_{D, D'}"] \arrow[ru, dashed, "\exists L_{D, D'}"] & \mathcal{C}(HD, HD')
\end{tikzcd}
\end{equation*}
again using the model structure on the category of groupoids.

To finish the construction of $(L, \lambda)$, we use the local fully faithfulness of $G$ to define 
\begin{equation*}
\lambda = G^{-1}( \eta \circ (\gamma L)^{-1} ).
\end{equation*}
The calculation
\begin{equation*}
(G, \gamma) \circ (L, \lambda) = (G \circ L, G \lambda \circ \gamma L) = (G \circ L, G G^{-1}( \eta \circ (\gamma L)^{-1} ) \circ \gamma L) = ( H, \eta )
\end{equation*}
demonstrates that the lower right triangle of (\ref{cotrflift}) commutes. To check that the upper left triangle commutes as well, we use the fact that the square (\ref{cotrflift}) commutes to compute
\begin{equation*}
G \phi = H \kappa \circ \eta K \circ (\gamma F)^{-1} = GL \kappa \circ G G^{-1} (\eta K \circ (\gamma F)^{-1}) = G( L \kappa \circ \lambda K),
\end{equation*}
giving the desired result
\begin{equation*}
(F, \phi) = (L \circ K, L \kappa \circ \lambda K) = (L, \lambda) \circ (K, \kappa),
\end{equation*}
by the local faithfulness of $G$.

Lastly, we show that $(L, \lambda)$ is a morphism by verifying that (\ref{coh4}) and (\ref{coh5}) commute for $\lambda$. Since $G$ locally is faithful, it suffices to check that these diagrams commute after $G$ is applied to them. But this follows directly from $\textbf{(1)} + \textbf{(3)} \Longrightarrow \textbf{(2)}$ of Lemma \ref{lem1}.
\end{proof}

\subsection{Factorization}

\begin{lem} \label{lem6}
Given a square of categories which commutes up to a natural isomorphism $\alpha : F H \Longrightarrow F K$
\begin{equation*}
\begin{tikzcd}[row sep=huge, column sep=huge]
\mathcal{A} \arrow[r, "K"] \arrow[d, swap, "H"] & \mathcal{B} \arrow[d, "F"] \\
\mathcal{B} \arrow[r, swap, "F"] \arrow[ru, Rightarrow, shorten >=20pt, shorten <=20pt, "\alpha"] & \mathcal{C}
\end{tikzcd}
\qquad = \qquad
\begin{tikzcd} [row sep=huge, column sep=huge]
\mathcal{A} \arrow[r, bend right=40, swap, "H"{name=H}] \arrow[r, bend left=40, "K"{name=K}] & \mathcal{B} \arrow[r, "F"] & \mathcal{C}
\arrow[from=H, to=K, dashed, Rightarrow, shorten >=5pt, shorten <=5pt, "\exists ! \beta"]
\end{tikzcd}
\end{equation*}
in which $F$ is an equivalence of categories, there exists a unique natural isomorphism $\beta : H \Longrightarrow K$ such that $F \beta = \alpha$.
\end{lem}

\begin{proof}
By hypothesis, there exists a functor $G : \mathcal{A} \longrightarrow \mathcal{B}$ and a natural isomorphism $\eta : \mathrm{id} \Longrightarrow GF$. For every $A$ in $\mathcal{A}$, the square
\begin{equation*}
\begin{tikzcd}[row sep=huge, column sep=huge]
HA \arrow[r, "\beta_{A}"] \arrow[d, swap, "\eta_{HA}"] & KA \arrow[d, "\eta_{KA}"] \\
GFHA \arrow[r, swap, "GF \beta_{A}"] & GFKA
\end{tikzcd}
\end{equation*}
must commute by naturality of $\eta$. Since $F \beta_{A} = \alpha_{A}$ is required as well, this leaves the composite
\begin{equation*}
H \xRightarrow{\eta H} GFH \xRightarrow{G \alpha} GFK \xRightarrow{(\eta K)^{-1}} K
\end{equation*}
as the only possible candidate for $\beta$. We see that the square
\begin{equation*}
\begin{tikzcd}[row sep=huge, column sep=huge]
FHA \arrow[r, "\alpha_{A}"] \arrow[d, swap, "F \eta_{HA}"] & FKA \arrow[d, "F \eta_{KA}"] \\
FGFHA \arrow[r, swap, "FG \alpha_{A}"] & FGFKA
\end{tikzcd}
\end{equation*}
commutes by naturality of $\eta$, as $\alpha_{A} = FF^{-1} \alpha_{A}$. This shows that our definition of $\beta$ indeed meets the requirement $F \beta = \alpha$. 
\end{proof}

\begin{lem}
Let $(F, \phi) : \mathcal{A} \longrightarrow \mathcal{C}$ be a morphism of bigroupoids. Then there exists a factorization 
\begin{equation*}
\mathcal{A} \xrightarrow{(G, \gamma)} \mathcal{B} \xrightarrow{(H, \eta)} \mathcal{C}
\end{equation*}
of $F$, where $G$ is a cofibration and $H$ is a strict trivial fibration.
\end{lem}

\begin{proof}
We define the 0-cells of $\mathcal{B}$ as the disjoint union of those of $\mathcal{A}$ and $\mathcal{C}$, so $\mathcal{B}_{0} = \mathcal{A}_{0} + \mathcal{C}_{0}$. We let $G : \mathcal{A}_{0} \longrightarrow \mathcal{B}_{0}$ be the inclusion map and we take $H = [F, \mathrm{id}] : \mathcal{B}_{0} \longrightarrow \mathcal{C}_{0}$.

To define the groupoids $\mathcal{B}(B, B')$, we factorize each $F_{A, A'} : \mathcal{A}(A, A') \longrightarrow \mathcal{C}(FA, FA')$  as
\begin{equation*}
\mathcal{A}(A, A') \xrightarrow{G_{A,A'}} \mathcal{B}(A, A') \xrightarrow{H_{A, A'}} \mathcal{C}(FA, FA'),
\end{equation*}
where $G_{A,A'}$ is a cofibration and $H_{A,A'}$ is a trivial fibration, using the model structure on the category of groupoids. For pairs of 0-cells of $\mathcal{B}$ not of the form $(A, A')$, we take (disjoint copies of) the groupoids in $\mathcal{C}$ corresponding to their image under $H$:
\begin{equation*}
\mathcal{B}(A, B') = \mathcal{C}(FA, B'), \qquad \mathcal{B}(B, A') = \mathcal{C}(B, FA'), \qquad \mathcal{B}(B, B') = \mathcal{C}(B, B').
\end{equation*}
The functor $H_{B, B'} : \mathcal{B}(B, B') \longrightarrow \mathcal{C}(HB, HB')$ is simply the identity in these last three cases.

We will now provide the functor $\mathbf{C}_{B, B', B''} : \mathcal{B}(B',B'') \times \mathcal{B}(B,B') \longrightarrow \mathcal{B}(B,B'')$ for a given triple of 0-cells $B$, $B'$, $B''$. Since $H_{B, B''} : \mathcal{B}(B, B'') \longrightarrow \mathcal{C}(HB, HB'')$ is a trivial fibration, it has a section $S_{B, B''} : \mathcal{C}(HB, HB'') \longrightarrow \mathcal{B}(B, B'')$. We define $\mathbf{C}_{B, B', B''}$ as the composite
\begin{equation*}
\mathcal{B}(B', B'') \times \mathcal{B}(B, B') \xrightarrow{H \times H} \mathcal{C}(HB', HB'') \times \mathcal{C}(HB, HB') \overset{\mathbf{C}}{\longrightarrow} \mathcal{C}(HB, HB'') \overset{S}{\longrightarrow} \mathcal{B}(B, B'').
\end{equation*}
Note that this makes the square
\begin{equation*}
\begin{tikzcd}[row sep=huge, column sep=huge]
\mathcal{B}(B', B'') \times \mathcal{B}(B, B') \arrow[r, "\mathbf{C}"] \arrow[d, swap, "H \times H"] & \mathcal{B}(B, B'') \arrow[d, "H"] \\
\mathcal{C}(HB', HB'') \times \mathcal{C}(HB, HB') \arrow[r, swap, "\mathbf{C}"] & \mathcal{C}(HB, HB'')
\end{tikzcd}
\end{equation*}
commute, which allows us to define $\eta = \mathrm{id}$.

Next, we define $\mathbf{a} = S \mathbf{a} H$. Since $H S \mathbf{a} H = \mathbf{a} H$ and $\eta = \mathrm{id}$, the diagram (\ref{coh4}) commutes for $\eta$. We use a similar definition for $\mathbf{l}$, $\mathbf{r}$, $\mathbf{e}$ and $\mathbf{i}$, so by the same argument the diagrams (\ref{coh5}) commute as well, hence $(H, \eta)$ is a morphism.

To show that $\mathcal{B}$ is a bigroupoid, we verify that the diagrams (\ref{coh1}), (\ref{coh2}) and (\ref{coh3}) commute. Since $H$ is locally faithful, these diagrams commute if and only if they commute after $H$ is applied to them. But this follows directly from $\textbf{(1)} \Longrightarrow \textbf{(2)}$ of Lemma \ref{lem2}.

To define $\gamma$, consider the square
\begin{equation} \label{psidef}
\begin{tikzcd}[row sep=huge, column sep=huge]
\mathcal{A}(A', A'') \times \mathcal{A}(A, A') \arrow[r, "G \circ \mathbf{C}"] \arrow[d, swap, "\mathbf{C} \circ (G \times G)"] & \mathcal{B}(GA, GA'') \arrow[d, "H"] \\
\mathcal{B}(GA, GA'') \arrow[r, swap, "H"] \arrow[ru, Rightarrow, shorten >=35pt, shorten <=35pt, "\phi \circ (\eta G)^{-1}"] & \mathcal{C}(FA, FA'')
\end{tikzcd}
\end{equation}
The calculation
\begin{equation*}
H \circ \mathbf{C} \circ (G \times G) \xRightarrow{(\eta G)^{-1}} \mathbf{C} \circ (H \times H) \circ (G \times G) = \mathbf{C} \circ (F \times F) \xRightarrow{\; \phi \;} F \circ \mathbf{C} = H \circ G \circ \mathbf{C}
\end{equation*}
shows that (\ref{psidef}) indeed commutes up to the natural isomorphism $\phi \circ (\eta G)^{-1}$. Since $H$ in (\ref{psidef}) is an equivalence of categories, Lemma \ref{lem6} provides us with a natural isomorphism
\begin{equation*}
\gamma (= \gamma_{A, A', A''}) : \mathbf{C} \circ ( G \times G ) \Longrightarrow G \circ \mathbf{C}
\end{equation*}
satisfying $H \gamma = \phi \circ (\eta G)^{-1}$. This means that we have indeed factored $(F, \phi)$ as $(H, \eta) \circ (G, \gamma)$.

To show that $(G, \gamma)$ is a morphism, we must verify that (\ref{coh4}) and (\ref{coh5}) commute for $\gamma$. Since $H$ is locally faithful, these diagrams commute if and only if they commute after $H$ is applied to them. But this follows directly from $\textbf{(1)} + \textbf{(3)} \Longrightarrow \textbf{(2)}$ of Lemma \ref{lem1}.
\end{proof}

\section{The trivial cofibration - fibration WFS}

The purpose of this section is to prove the following proposition.

\begin{prop} \label{prop}
The trivial cofibrations and fibrations form a weak factorization system.
\end{prop}

\subsection{Lifting property}

\begin{lem} \label{lem3}
Given a triangle of groupoids that commutes up to a natural isomorphism $\beta : H \Longrightarrow GF$
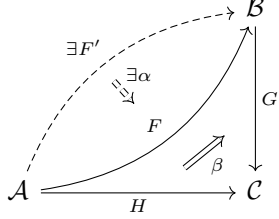
\begin{equation*}
\begin{tikzcd}[row sep=large, column sep=large]
& & \mathcal{B} \arrow[dd, "G"{name=G}] \\
& & \\
\mathcal{A} \arrow[rr, swap, "H"{name=H}] \arrow[rruu, bend right, "F"{name=F}] \arrow[rruu, dashed, bend left, "\exists F'"{name=F'}] &  & \mathcal{C}
\arrow[Rightarrow, dashed, from = F', to=F, shorten >=5pt, shorten <=10pt, "\exists \alpha"]
\arrow[swap, Rightarrow, from=H, to=G, shorten >=15pt, shorten <=15pt, "\beta"]
\end{tikzcd}
\end{equation*}
and in which $G$ is a fibration, there exists a functor $F'$ making the triangle commute, along with a natural isomorphism $\alpha : F' \Longrightarrow F$ such that $G \alpha = \beta$.
\end{lem}

\begin{proof}
For every object $A$ of $\mathcal{A}$, there exists an object $B_{A}$ of $\mathcal{B}$ and an arrow $\alpha_{A} : B_{A} \longrightarrow FA$ such that $GB_{A} = HA$ and $G \alpha_{A} = \beta_{A}$, since $G$ is a fibration. Define $F' A = B_{A}$ and $F (f : A \longrightarrow A' ) = \alpha_{A'}^{-1} \circ Ff \circ \alpha_{A}$.
\end{proof}

\begin{lem} \label{lem5}
Given a square of categories which commutes up to a natural isomorphism $\alpha : H G \Longrightarrow K G$
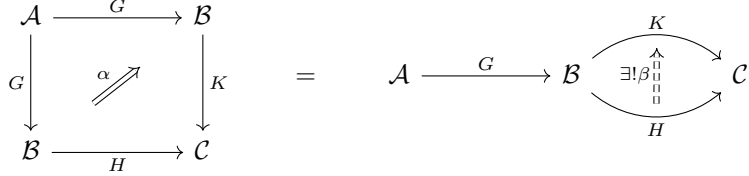
\begin{equation*}
\begin{tikzcd}[row sep=huge, column sep=huge]
\mathcal{A} \arrow[r, "G"] \arrow[d, swap, "G"] & \mathcal{B} \arrow[d, "K"] \\
\mathcal{B} \arrow[r, swap, "H"] \arrow[ru, Rightarrow, shorten >=20pt, shorten <=20pt, "\alpha"] & \mathcal{C}
\end{tikzcd}
\qquad = \qquad
\begin{tikzcd} [row sep=huge, column sep=huge]
\mathcal{A} \arrow[r, "G"] & \mathcal{B} \arrow[r, bend right=40, swap, "H"{name=H}] \arrow[r, bend left=40, "K"{name=K}] & \mathcal{C} 
\arrow[from=H, to=K, dashed, Rightarrow, shorten >=5pt, shorten <=5pt, "\exists ! \beta"]
\end{tikzcd}
\end{equation*}
in which $G$ is an equivalence of categories, there exists a unique natural isomorphism $\beta : H \Longrightarrow K$ such that $\beta G = \alpha$.
\end{lem}

\begin{proof}
By hypothesis, there exists a functor $F : \mathcal{B} \longrightarrow \mathcal{A}$ and a natural isomorphism $\eta : \mathrm{id} \Longrightarrow GF$. For every $B$ in $\mathcal{B}$, the square
\begin{equation*}
\begin{tikzcd}[row sep=huge, column sep=huge]
H B \arrow[r, "\beta_{B}"] \arrow[d, swap, "H \eta_{B}"] & K C \arrow[d, "K \eta_{B}"] \\
HGFB \arrow[r, swap, "\beta_{GFB}"] & KGFB
\end{tikzcd}
\end{equation*}
must commute by naturality of $\beta$. Since $\beta_{GFB} = \alpha_{FB}$ is required as well, this leaves the composite
\begin{equation*}
H \xRightarrow{H \eta} H G F \xRightarrow{\alpha F} K G F \xRightarrow{(K \eta)^{-1}} K
\end{equation*}
as the only possible candidate for $\beta$. We see that the square
\begin{equation*}
\begin{tikzcd}[row sep=huge, column sep=huge]
H GA \arrow[r, "\alpha_{A}"] \arrow[d, swap, "H \eta_{GA}"] & K GA \arrow[d, "K \eta_{GA}"] \\
H GFGA \arrow[r, swap, "\alpha_{FGA}"] & K GFGA
\end{tikzcd}
\end{equation*}
commutes by naturality of $\alpha$, as $H \eta_{GA} = H G G^{-1} \eta_{GA}$ and $K \eta_{GA} = K G G^{-1} \eta_{GA}$. This shows that our definition of $\beta$ indeed meets the requirement $\beta G = \alpha$. 
\end{proof}

\begin{lem} \label{muinvlem}
In any diagram of categories
\begin{equation*}
\begin{tikzcd} [row sep=huge, column sep=huge]
\mathcal{A} \arrow[r, bend right=40, swap, "G"{name=G}] \arrow[r, bend left=40, "F"{name=F}] & \mathcal{B} \arrow[r, bend right=40, swap, "K"{name=K}] \arrow[r, bend left=40, "H"{name=H}] & \mathcal{C} 
\arrow[from=H, to=K, shift left=2, Rightarrow, shorten >=5pt, shorten <=5pt, "\beta"]
\arrow[from=H, to=K, swap, shift right=2, Rightarrow, shorten >=5pt, shorten <=5pt, "\alpha"]
\arrow[from=F, to=G, swap, Rightarrow, shorten >=5pt, shorten <=5pt, "\mu"]
\end{tikzcd}
\end{equation*}
with natural transformations $\alpha, \beta : H \Longrightarrow K$ and a natural isomorphism $\mu : F \Longrightarrow G$, the equality $\alpha F = \beta F$ holds if and only if the equality $\alpha G = \beta G$ holds.
\end{lem}

\begin{proof}
This follows from the equations
\begin{equation*}
K \mu \circ \alpha F = \alpha G \circ H \mu \qquad \text{and} \qquad K \mu \circ \beta F = \beta G \circ H \mu
\end{equation*}
and the fact that $\mu$ is invertible.
\end{proof}

\begin{cor} \label{cor1}
Let $(F, \phi) : \mathcal{A} \longrightarrow \mathcal{B}$ be an incoherent morphism between (possibly incoherent) bigroupoids. Suppose furthermore that for every pair of 0-cells $A$, $A'$ of $\mathcal{A}$, two endofunctors $G_{A, A'}, H_{A, A'} : \mathcal{A}(A, A') \longrightarrow \mathcal{A}(A, A')$ are given which are naturally isomorphic $\mu_{A, A'} : G_{A, A'} \Longrightarrow H_{A, A'}$. Then the diagrams (\ref{coh4}) and (\ref{coh5}) commute for $\phi G$ if and only if they commute for $\phi H$.
\end{cor}

\begin{proof}
This is a direct application of Lemma \ref{muinvlem}.
\end{proof}

\begin{lem} \label{surlift}
Given a commutative square
\begin{equation} \label{lift3}
\begin{tikzcd}[row sep=huge, column sep=huge]
\mathcal{A} \arrow[r, "{(F, \phi)}"] \arrow[d, swap, "{(K, \kappa)}"] & \mathcal{B} \arrow[d, "{(G, \gamma)}"] \\
\mathcal{D} \arrow[r, swap, "{(H, \eta)}"] \arrow[ru, dashed, "{\exists (L, \lambda)}"] & \mathcal{C}
\end{tikzcd}
\end{equation}
in which $K$ is a trivial cofibration which is surjective on 0-cells and $G$ is a fibration, there exists a diagonal filler $L$, as indicated in the diagram.
\end{lem}

\begin{proof}
Let $L : \mathcal{D}_{0} \longrightarrow \mathcal{B}_{0}$ to be the unique function that makes the diagram
\begin{equation*}
\begin{tikzcd}[row sep=huge, column sep=huge]
\mathcal{A}_{0} \arrow[r, "F"] \arrow[d, swap, "K"] & \mathcal{B}_{0} \arrow[d, "G"] \\
\mathcal{D}_{0} \arrow[r, swap, "H"] \arrow[ru, dashed, "\exists ! L"] & \mathcal{C}_{0}
\end{tikzcd}
\end{equation*}
commute. This function exists because $K : \mathcal{A}_{0} \longrightarrow \mathcal{D}_{0}$ is bijective.

Given two 0-cells $D = KA$ and $D'= KA'$ in $\mathcal{D}$, we construct the functor
\begin{equation*}
L( = L_{D, D'}) : \mathcal{D}(D, D') \longrightarrow \mathcal{B}(LD, LD')
\end{equation*}
by taking a diagonal
\begin{equation*}
\begin{tikzcd}[row sep=huge, column sep=huge]
\mathcal{A}(A, A') \arrow[r, "F"] \arrow[d, swap, "K"] & \mathcal{B}(LD, LD') \arrow[d, "G"] \\
\mathcal{D}(D, D') \arrow[r, swap, "H"] \arrow[ru, dashed, "\exists L"] & \mathcal{C}(HD, HD')
\end{tikzcd}
\end{equation*}
which exists by the model structure on the category of groupoids.

To define $\lambda$, consider the square
\begin{equation} \label{lamdef}
\begin{tikzcd}[row sep=huge, column sep=huge]
\mathcal{A}(A', A'') \times \mathcal{A}(A, A') \arrow[r, "K \times K"] \arrow[d, swap, "K \times K"] & \mathcal{D}(D', D'') \times \mathcal{D}(D, D') \arrow[d, "L \circ \mathbf{C}"] \\
\mathcal{D}(D', D'') \times \mathcal{D}(D, D') \arrow[r, swap, "\mathbf{C} \circ (L \times L)"] \arrow[ru, Rightarrow, shorten >=40pt, shorten <=40pt, "(L \kappa)^{-1} \circ \phi"] & \mathcal{B}(LA, LA'')
\end{tikzcd}
\end{equation}
The calculation
\begin{equation*}
\mathbf{C} \circ (L \times L) \circ (K \times K) = \mathbf{C} \circ (F \times F) \xRightarrow{\; \phi \;} F \circ \mathbf{C} = L \circ K \circ \mathbf{C} \xRightarrow{(L \kappa)^{-1}} L \circ \mathbf{C} \circ (K \times K)
\end{equation*}
shows that (\ref{lamdef}) indeed commutes up to the natural isomorphism $(L \kappa)^{-1} \circ \phi$. Since $K \times K$ in (\ref{lamdef}) is an equivalence of categories, Lemma \ref{lem5} provides us with a natural isomorphism
\begin{equation*}
\lambda (= \lambda_{D, D', D''}) : \mathbf{C} \circ (L \times L) \Longrightarrow L \circ \mathbf{C}
\end{equation*}
satisfying $\lambda K = (L \kappa)^{-1} \circ \phi$.

We make the necessary verifications. The left upper triangle of (\ref{lift3}) commutes, since
\begin{equation*}
(L, \lambda) \circ (K, \kappa) = (L \circ K, L \kappa \circ \lambda K) = (F, \phi),
\end{equation*}
as $\lambda K = (L \kappa)^{-1} \circ \phi$. We can also compute
\begin{equation*}
(G \lambda \circ \gamma L)K = G \lambda K \circ \gamma LK = G ( (L \kappa)^{-1} \circ \phi) \circ \gamma F = (H \kappa)^{-1} \circ G \phi \circ \gamma F = \eta K,
\end{equation*}
using $\lambda K = (L \kappa)^{-1} \circ \phi$ as well as the commutativity of the square (\ref{lift3}). Hence
\begin{equation*}
(G, \gamma) \circ (L, \lambda) = (G \circ L, G \lambda \circ \gamma L) = (H, \eta)
\end{equation*}
by the uniqueness requirement of Lemma \ref{lem5}, so the lower right triangle of (\ref{lift3}) commutes as well.

Lastly, we check that the coherence diagrams (\ref{coh4}) and (\ref{coh5}) commute for $\lambda$. Note that for each pair of 0-cells $D$, $D'$ of $\mathcal{D}$, there exists a functor 
\begin{equation*}
T_{D, D'} : \mathcal{D}(D, D') \longrightarrow \mathcal{A}(A, A')
\end{equation*}
and a natural isomorphism
\begin{equation*}
\alpha_{D, D'} : \mathrm{id} \Longrightarrow K_{A, A'} \circ T_{D, D'},
\end{equation*}
as each $K_{A, A'}$ is an equivalence of categories. Since $(L, \lambda) \circ (K, \kappa) = (F, \phi)$, it follows that the diagrams (\ref{coh4}) and (\ref{coh5}) commute for $\lambda K$, by $\textbf{(2)} + \textbf{(3)} \Longrightarrow \textbf{(1)}$ of Lemma \ref{lem1}. In particular, they commute for $\lambda K T$. But then they commute for $\lambda$ by Corollary \ref{cor1}.
\end{proof}

\begin{lem} \label{isolift}
Given a commutative square
\begin{equation}\label{eq3}
\begin{tikzcd}[row sep=huge, column sep=huge]
\mathcal{A} \arrow[r, "{(F, \phi)}"] \arrow[d, swap, "{(K, \mathrm{id})}"] & \mathcal{B} \arrow[d, "{(G, \mathrm{id})}"] \\
\mathcal{C} \arrow[r, swap, "\mathrm{id}"] \arrow[ru, dashed, "{\exists (L, \lambda)}"] & \mathcal{C}
\end{tikzcd}
\end{equation}
in which $K$ is a strict trivial cofibration, which is also a local isomorphism and $G$ is a strict fibration, there exists a diagonal filler $L$, as indicated in the diagram.
\end{lem}

\begin{proof}
We build $(L, \lambda)$ in three stages, each time `correcting' the previous stage. The morphism $(L^{(1)}, \lambda^{(1)})$ will make the upper-left triangle commute. In addition to this, $(L^{(2)}, \lambda^{(2)})$ will make the diagram commute on the level of 0-cells. And  finally $(L^{(3)}, \lambda^{(3)}) = (L, \lambda)$ will make the entire diagram commute.

\textbf{Stage 1.} We construct a left inverse $(T, \tau) : \mathcal{C} \longrightarrow \mathcal{A}$ of $K$. Since $K$ is a trivial cofibration, there exists a function $T : \mathcal{C}_{0} \longrightarrow \mathcal{A}_{0}$ such that $TK = \mathrm{id}$ and for every 0-cell $C$ of $\mathcal{C}$, there exists a 1-cell $p_{C} : C \longrightarrow KTC$. Whenever $KTC = C$, we choose $p_{C} = 1_{C}$. We define members $P_{C, C'}$ of a $\mathcal{C}_{0} \times \mathcal{C}_{0}$-indexed family of functors by:
\begin{itemize}
\item{
$\begin{tikzcd}[column sep=huge]
\mathcal{C}(C, C') \arrow[r, "p_{C'} * ( - * p_{C}^{*})"]  & \mathcal{C}(KTC, KTC')
\end{tikzcd}$, if at least one of $C$, $C'$ does not lie in the image of $K$;}
\item{$\begin{tikzcd}[column sep=huge]
\mathcal{C}(C, C') \arrow[r, "\mathrm{id}"]  & \mathcal{C}(KTC, KTC')
\end{tikzcd}$, if both $C$ and $C'$ lie in the image of $K$.}
\end{itemize}
We take $T_{C, C'} = K^{-1}_{TC, TC'} \circ P_{C, C'}$. 

The natural isomorphism
\begin{equation*}
\tau (= \tau_{C, C', C''}) : \mathbf{C} \circ (T \times T) \Longrightarrow T \circ \mathbf{C}
\end{equation*}
is given by the diagram
\begin{equation} \label{tau}
\begin{tikzcd}[row sep=huge, column sep=huge]
\mathcal{C}(C', C'') \times \mathcal{C}(C, C') \arrow[r, "\mathbf{C}"] \arrow[d, swap, "P \times P"] & \mathcal{C}(C, C'') \arrow[d, "P"] \\
\mathcal{C}(KTC', KTC'') \times \mathcal{C}(KTC, KTC') \arrow[r, "\mathbf{C}"] \arrow[d, swap, "K^{-1} \times K^{-1}"] \arrow[ru, Rightarrow, shorten >=50pt, shorten <=50pt, "\mathbf{x}"] & \mathcal{C}(KTC, KTC'') \arrow[d, "K^{-1}"] \\
\mathcal{A}(TC', TC'') \times \mathcal{A}(TC, TC') \arrow[r, swap, "\mathbf{C}"] \arrow[ru, Rightarrow, shorten >=47pt, shorten <=47pt, "\mathrm{id}"] & \mathcal{A}(TC, TC'')
\end{tikzcd}
\end{equation}
In (\ref{tau}), $\mathbf{x}( = \mathbf{x}_{C, C', C''})$ is the canonical isomorphism (see Definition \ref{fordiagdef}). The diagrams (\ref{coh4}) and (\ref{coh5}) commute for $\tau$ by Theorem \ref{fordiagthm} since $\mathbf{x}$ is canonical and $K$ is a strict local isomorphism. Define $(L^{(1)}, \lambda^{(1)}) = (F, \phi) \circ (T, \tau)$ and note that $(L^{(1)}, \lambda^{(1)}) \circ (K, \mathrm{id}) = (F, \phi)$, as $(T, \tau) \circ (K, \mathrm{id}) = \mathrm{id}$ by construction.

\textbf{Stage 2.} Since $G$ is a fibration, there exists a function $L^{(2)} : \mathcal{C}_{0} \longrightarrow \mathcal{B}_{0}$ such that $L^{(2)}K = L^{(1)}K$, $G L^{(2)} = \mathrm{id}$ and for every 0-cell $C$ of $\mathcal{C}$, there exists a 1-cell $q_{C} : L^{(2)} C \longrightarrow L^{(1)}C$ satisfying $G q_{C} = p_{C}$. Whenever $KTC = C$, we choose $q_{C} = 1_{L^{(2)}C}$. We define members $Q_{C, C'}$ of a $\mathcal{C}_{0} \times \mathcal{C}_{0}$-indexed family of functors by:
\begin{itemize}
\item{
$\begin{tikzcd}[column sep=huge]
\mathcal{B}(L^{(1)}C, L^{(1)}C') \arrow[r, "q_{C'}^{*} * ( - * q_{C})"]  & \mathcal{B}(L^{(2)}C, L^{(2)}C')
\end{tikzcd}$, if at least one of $C$, $C'$ does not lie in the image of $K$;}
\item{$\begin{tikzcd}[column sep=huge]
\mathcal{B}(L^{(1)}C, L^{(1)}C') \arrow[r, "\mathrm{id}"]  & \mathcal{B}(L^{(2)}C, L^{(2)}C')
\end{tikzcd}$, if both $C$ and $C'$ lie in the image of $K$.}
\end{itemize}
We take $L_{C, C'}^{(2)} = Q_{C, C'} \circ L_{C, C'}^{(1)}$. 

The natural isomorphism
\begin{equation*}
\lambda^{(2)} (= \lambda^{(2)}_{C, C', C''}) : \mathbf{C} \circ (L^{(2)} \times L^{(2)}) \Longrightarrow L^{(2)} \circ \mathbf{C}
\end{equation*}
is given by the diagram
\begin{equation} \label{lam2}
\begin{tikzcd}[row sep=huge, column sep=huge]
\mathcal{C}(C', C'') \times \mathcal{C}(C, C') \arrow[r, "\mathbf{C}"] \arrow[d, swap, "L^{(1)} \times L^{(1)}"] & \mathcal{C}(C, C'') \arrow[d, "L^{(1)}"] \\
\mathcal{B}(L^{(1)}C', L^{(1)}C'') \times \mathcal{B}(L^{(1)}C, L^{(1)}C') \arrow[r, "\mathbf{C}"] \arrow[d, swap, "Q \times Q"] \arrow[ru, Rightarrow, shorten >=45pt, shorten <=45pt, "\lambda^{(1)}"] & \mathcal{B}(L^{(1)}C, L^{(1)}C'') \arrow[d, "Q"] \\
\mathcal{B}(L^{(2)}C', L^{(2)}C'') \times \mathcal{B}(L^{(2)}C, L^{(2)}C') \arrow[r, swap, "\mathbf{C}"] \arrow[ru, Rightarrow, shorten >=44pt, shorten <=44pt, "\mathbf{y}"] & \mathcal{B}(L^{(2)}C, L^{(2)}C'')
\end{tikzcd}
\end{equation}
In (\ref{lam2}), $\mathbf{y}( = \mathbf{y}_{C, C', C''})$ is the canonical isomorphism. By Theorem \ref{phifordiagthm} applied to $(L^{(1)}, \lambda^{(1)})$, the diagrams (\ref{coh4}) and (\ref{coh5}) commute for $\lambda^{(2)}$. Note that $(L^{(2)}, \lambda^{(2)}) \circ (K, \mathrm{id}) = (F, \phi)$, as $(L^{(2)}, \lambda^{(2)}) \circ (K, \mathrm{id}) = (L^{(1)}, \lambda^{(1)}) \circ (K, \mathrm{id})$ by construction.

\textbf{Stage 3.} We now modify $(L^{(2)}, \lambda^{(2)})$ to get the desired morphism $(L, \lambda)$. On the level of 0-cells, we make no changes, meaning that $L = L^{(2)} : \mathcal{C}_{0} \longrightarrow \mathcal{B}_{0}$. The need to modify $(L^{(2)}, \lambda^{(2)})$ arises because the triangle
\begin{equation} \label{eq4}
\begin{tikzcd}
& & \mathcal{B}(LC, LC') \arrow[dd, "G"] \\
& & {}\\
\mathcal{C}(C, C') \arrow[rr, swap, "\mathrm{id}"] \arrow[rruu, "L^{(2)}"] & {} \arrow[ru, Rightarrow, shorten >=20pt, shorten <=20pt, "\mathbf{z}"] & \mathcal{C}(C, C')
\end{tikzcd}
\end{equation}
will in general only commute up to a canonical isomorphism $\mathbf{z} (= \mathbf{z}_{C, C'})$. Indeed, let us define members $R_{C, C'}$ of a $\mathcal{C}_{0} \times \mathcal{C}_{0}$-indexed family of functors by:
\begin{itemize}
\item{
$\begin{tikzcd}[column sep=huge]
\mathcal{C}(KTC, KTC') \arrow[r, "p_{C'}^{*} * ( - * p_{C})"] & \mathcal{C}(C, C') 
\end{tikzcd}$, if at least one of $C$, $C'$ does not lie in the image of $K$;}
\item{$\begin{tikzcd}[column sep=huge]
\mathcal{C}(KTC, KTC') \arrow[r, "\mathrm{id}"]  & \mathcal{C}(C, C')
\end{tikzcd}$, if both $C$ and $C'$ lie in the image of $K$.}
\end{itemize}
Using the relations $G q_{C} = p_{C}$, $G q_{C'} = p_{C'}$ and the strictness of $G$, one easily verifies
\begin{equation} \label{GQ=RG}
G_{L^{(2)} C, L^{(2)} C'} \circ Q_{C, C'} = R_{C, C'} \circ G_{L^{(1)} C, L^{(1)} C'}.
\end{equation}
Then, with $G$ and $L^{(2)}$ as in (\ref{eq4}),
\begin{equation} \label{GL}
G \circ L^{(2)} = G \circ Q \circ L^{(1)} =  G \circ Q \circ F \circ T = G \circ Q \circ F \circ K^{-1} \circ P,
\end{equation}
all by definition. Now using $G \circ Q = R \circ G$ (by (\ref{GQ=RG})) and $G \circ F = K$ (by (\ref{eq3})), we find that (\ref{GL}) is equal to
\begin{equation*}
R \circ G \circ F \circ K^{-1} \circ P = R \circ K \circ K^{-1} \circ P = R \circ P
\end{equation*}
and clearly there exists a canonical isomorphism $\mathbf{z} : \mathrm{id} \Longrightarrow R \circ P$.

If both $C$ and $C'$ lie in the image of $K$, then $\mathbf{z}$ is the identity and we define $L_{C, C'} = L^{(2)}_{C, C'}$ and $\alpha_{C, C'} = \mathrm{id} : L_{C, C'} \Longrightarrow L^{(2)}_{C, C'}$. In all other cases we apply Lemma \ref{lem3} to obtain a functor $L_{C, C'} : \mathcal{C}(C, C') \longrightarrow \mathcal{B}(LC, LC')$ which does make the triangle (\ref{eq4}) commute, together with a natural isomorphism $\alpha_{C, C'} : L_{C, C'} \Longrightarrow L^{(2)}_{C, C'}$ satisfying $G \alpha = \mathbf{z}$. We define $\lambda$ as the natural isomorphism
\begin{equation*}
\begin{tikzcd}[row sep=huge, column sep=huge]
\mathcal{C}(C', C'') \times \mathcal{C}(C, C') \arrow[r, "\mathbf{C}"] \arrow[d, bend right=55, swap, "L \times L"{name=LL}] \arrow[d, bend left=55, "L^{(2)} \times L^{(2)}"{name=LLB}] & \mathcal{C}(C, C'') \arrow[d, swap, bend right=55, "L^{(2)}"{name=LB}] \arrow[d, bend left=55, "L"{name=L}] \\
\mathcal{B}(LC', LC'') \times \mathcal{B}(LC, LC') \arrow[r, swap, "\mathbf{C}"] \arrow[ru, Rightarrow, shorten >=40pt, shorten <=40pt, "\lambda^{(2)}" pos=0.6] & \mathcal{B}(LC, LC'')
\arrow[from=LL, to=LLB, Rightarrow, shorten >=10pt, shorten <=10pt, "\alpha \times \alpha"]
\arrow[from=LB, to=L, Rightarrow, shorten >=10pt, shorten <=10pt, "\alpha^{-1}"]
\end{tikzcd}
\end{equation*}
Note that that this choice of $(L, \lambda)$ gives $(L, \lambda) \circ (K, \mathrm{id}) = (L^{(2)}, \lambda^{(2)}) \circ (K, \mathrm{id}) = (F, \phi)$ and also ensures that the lower right triangle of (\ref{eq3}) commutes on the level of 0-, 1- and 2-cells.

To verify that the coherence diagram (\ref{coh4}) commutes for $\lambda$, consider the following diagram, whose perimeter is exactly (\ref{coh4}):
\begin{equation*}
\begin{tikzcd}[row sep=large, column sep=huge]
\cdot \arrow[rr, "\lambda * \mathrm{id}"] \arrow[ddd, swap, "\mathbf{a}"] \arrow[dr, "(\alpha * \alpha) * \alpha" description] & & \cdot \arrow[rr, "\lambda"] \arrow[d, swap, "\alpha * \alpha" description] & & \cdot \arrow[ddd, "L\mathbf{a}"] \arrow[dl, swap, "\alpha" description] \\
& \cdot \arrow[r, "\lambda^{(2)} * \mathrm{id}"] \arrow[d, swap, "\mathbf{a}"] & \cdot \arrow[r, "\lambda^{(2)}"] & \cdot \arrow[d, "L^{(2)} \mathbf{a}"] & \\
& \cdot \arrow[r, swap, "\mathrm{id} * \lambda^{(2)}"] & \cdot \arrow[r, swap, "\lambda^{(2)}"] & \cdot & \\
\cdot \arrow[rr, swap, "\mathrm{id} * \lambda"] \arrow[ur, swap, "\alpha * (\alpha * \alpha)" description] & & \cdot \arrow[rr, swap, "\lambda"] \arrow[u, "\alpha * \alpha" description] & & \cdot \arrow[ul, "\alpha" description]
\end{tikzcd}
\end{equation*}
The innermost rectangle is simply diagram (\ref{coh4}) for $\lambda^{(2)}$, which commutes because $(L^{(2)}, \lambda^{(2)})$ is a morphism; the leftmost square commutes by naturality of $\mathbf{a}$; the rightmost square commutes by naturality of $\alpha$ and all other `squares' in the diagram commute by definition of $\lambda$.

All that remains to show is that $G \lambda = \mathrm{id}$. Expand the definition of $\lambda$ to get
\begin{equation*}
\begin{tikzcd}[row sep=huge, column sep=huge]
\cdot \arrow[r, "\mathbf{C}"] \arrow[d, swap, "L \times L"] & \cdot \arrow[d, "L"] \\
\cdot \arrow[r, "\mathbf{C}"] \arrow[d, swap, "G \times G"] \arrow[ru, Rightarrow, shorten >=20pt, shorten <=20pt, "\lambda"] & \cdot \arrow[d, "G"] \\
\cdot \arrow[r, swap, "\mathbf{C}"] \arrow[ru, Rightarrow, shorten >=20pt, shorten <=20pt, "\mathrm{id}"] & \cdot
\end{tikzcd}
\qquad = \qquad
\begin{tikzcd}[row sep=huge, column sep=huge]
\cdot \arrow[r, "\mathbf{C}"] \arrow[d, bend right=55, swap, "L \times L"{name=LL}] \arrow[d, bend left=55, "{}"{name=LLB}] & \cdot \arrow[d, swap, bend right=55, "{}"{name=LB}] \arrow[d, bend left=55, "L"{name=L}] \\
\cdot \arrow[d, swap, "G \times G"] \arrow[r, "\mathbf{C}"] \arrow[ru, Rightarrow, shorten >=20pt, shorten <=20pt, "\lambda^{(2)}" pos=0.6] & \cdot \arrow[d, "G"] \\
\cdot \arrow[r, swap, "\mathbf{C}"] \arrow[ru, Rightarrow, shorten >=20pt, shorten <=20pt, "\mathrm{id}"] & \cdot
\arrow[from=LL, to=LLB, Rightarrow, shorten >=10pt, shorten <=10pt, "\alpha \times \alpha"]
\arrow[from=LB, to=L, Rightarrow, shorten >=10pt, shorten <=10pt, "\alpha^{-1}"]
\end{tikzcd}
\end{equation*}
Since $G \alpha = \mathbf{z}$, this is the same as
\begin{equation}\label{eq1}
\begin{tikzcd}[row sep=huge, column sep=huge]
\cdot \arrow[r, "\mathbf{C}"] \arrow[d] \arrow[dd, bend right=100, swap, "\mathrm{id}"{name=A}] & \cdot \arrow[d] \arrow[dd, bend left=100, "\mathrm{id}"{name=B}] \\
\cdot \arrow[r, "\mathbf{C}"] \arrow[d] \arrow[ru, Rightarrow, shorten >=20pt, shorten <=20pt, "\lambda^{(2)}"] & \cdot \arrow[d] \\
\cdot \arrow[r, swap, "\mathbf{C}"] \arrow[ru, Rightarrow, shorten >=20pt, shorten <=20pt, "\mathrm{id}"] & \cdot
\arrow[from=A, to=2-1, Rightarrow, shorten >=8pt, shorten <=8pt, "\mathbf{z} \times \mathbf{z}"] \arrow[from=2-2, to=B, Rightarrow, shorten >=8pt, shorten <=8pt, "\mathbf{z}^{-1}"]
\end{tikzcd}
\end{equation}
Now consider the two cental squares of (\ref{eq1}):
\begin{equation} \label{botsq}
\begin{tikzcd}[row sep=huge, column sep=huge]
\cdot \arrow[r, "\mathbf{C}"] \arrow[d, swap, "L^{(2)} \times L^{(2)}"] & \cdot \arrow[d, "L^{(2)}"] \\
\cdot \arrow[r, "\mathbf{C}"] \arrow[d, swap, "G \times G"] \arrow[ru, Rightarrow, shorten >=20pt, shorten <=20pt, "\lambda^{(2)}"] & \cdot \arrow[d, "G"] \\
\cdot \arrow[r, swap, "\mathbf{C}"] \arrow[ru, Rightarrow, shorten >=20pt, shorten <=20pt, "\mathrm{id}"] & \cdot
\end{tikzcd}
\qquad = \qquad
\begin{tikzcd}[row sep=huge, column sep=huge]
\cdot \arrow[r, "\mathbf{C}"] \arrow[d, swap, "L^{(1)} \times L^{(1)}"] & \cdot \arrow[d, "L^{(1)}"] \\
\cdot \arrow[r, "\mathbf{C}"] \arrow[d, swap, "Q \times Q"] \arrow[ru, Rightarrow, shorten >=20pt, shorten <=20pt, "\lambda^{(1)}"] & \cdot \arrow[d, "Q"] \\
\cdot \arrow[r, "\mathbf{C}"] \arrow[d, swap, "G \times G"] \arrow[ru, Rightarrow, shorten >=20pt, shorten <=20pt, "\mathbf{y}"] & \cdot \arrow[d, "G"] \\
\cdot \arrow[r, swap, "\mathbf{C}"] \arrow[ru, Rightarrow, shorten >=20pt, shorten <=20pt, "\mathrm{id}"] & \cdot
\end{tikzcd}
\qquad = \qquad
\begin{tikzcd}[row sep=huge, column sep=huge]
\cdot \arrow[r, "\mathbf{C}"] \arrow[d, swap, "L^{(1)} \times L^{(1)}"] & \cdot \arrow[d, "L^{(1)}"] \\
\cdot \arrow[r, "\mathbf{C}"] \arrow[d, swap, "G \times G"] \arrow[ru, Rightarrow, shorten >=20pt, shorten <=20pt, "\lambda^{(1)}"] & \cdot \arrow[d, "G"] \\
\cdot \arrow[r, "\mathbf{C}"] \arrow[d, swap, "R \times R"] \arrow[ru, Rightarrow, shorten >=20pt, shorten <=20pt, "\mathrm{id}"] & \cdot \arrow[d, "R"] \\
\cdot \arrow[r, swap, "\mathbf{C}"] \arrow[ru, Rightarrow, shorten >=20pt, shorten <=20pt, "\mathbf{w}"] & \cdot
\end{tikzcd}
\end{equation}
The first and second diagrams of (\ref{botsq}) are equal by definition of $(L^{(2)}, \lambda^{(2)})$. In the third diagram, $\mathbf{w}$ is the canonical isomorphism. The bottom two squares in the second diagram of (\ref{botsq}) and the bottom two squares in the third diagram of (\ref{botsq}) both represent a canonical isomorphism, so they must be equal. Using the definition of $(L^{(1)}, \lambda^{(1)})$ and applying $(G, \mathrm{id}) \circ (F, \phi) = (K, \mathrm{id})$, we find that (\ref{botsq}) is equal to
\begin{equation}\label{eq2}
\begin{tikzcd}[row sep=huge, column sep=huge]
\cdot \arrow[r, "\mathbf{C}"] \arrow[d, swap, "P \times P"] & \cdot \arrow[d, "P"] \\
\cdot \arrow[r, "\mathbf{C}"] \arrow[d, swap, "K^{-1} \times K^{-1}"] \arrow[ru, Rightarrow, shorten >=20pt, shorten <=20pt, "\mathbf{x}"] & \cdot \arrow[d, "K^{-1}"] \\
\cdot \arrow[r, "\mathbf{C}"] \arrow[d, swap, "F \times F"] \arrow[ru, Rightarrow, shorten >=20pt, shorten <=20pt, "\mathrm{id}"] & \cdot \arrow[d, "F"] \\
\cdot \arrow[r, "\mathbf{C}"] \arrow[d, swap, "G \times G"] \arrow[ru, Rightarrow, shorten >=20pt, shorten <=20pt, "\phi"] & \cdot \arrow[d, "G"] \\
\cdot \arrow[r, "\mathbf{C}"] \arrow[d, swap, "R \times R"] \arrow[ru, Rightarrow, shorten >=20pt, shorten <=20pt, "\mathrm{id}"] & \cdot \arrow[d, "R"] \\
\cdot \arrow[r, swap, "\mathbf{C}"] \arrow[ru, Rightarrow, shorten >=20pt, shorten <=20pt, "\mathbf{w}"] & \cdot
\end{tikzcd}
\qquad = \qquad
\begin{tikzcd}[row sep=huge, column sep=huge]
\cdot \arrow[r, "\mathbf{C}"] \arrow[d, swap, "P \times P"] & \cdot \arrow[d, "P"] \\
\cdot \arrow[r, "\mathbf{C}"] \arrow[d, swap, "K^{-1} \times K^{-1}"] \arrow[ru, Rightarrow, shorten >=20pt, shorten <=20pt, "\mathbf{x}"] & \cdot \arrow[d, "K^{-1}"] \\
\cdot \arrow[r, "\mathbf{C}"] \arrow[d, swap, "K \times K"] \arrow[ru, Rightarrow, shorten >=20pt, shorten <=20pt, "\mathrm{id}"] & \cdot \arrow[d, "K"] \\
\cdot \arrow[r, "\mathbf{C}"] \arrow[d, swap, "R \times R"] \arrow[ru, Rightarrow, shorten >=20pt, shorten <=20pt, "\mathrm{id}"] & \cdot \arrow[d, "R"] \\
\cdot \arrow[r, swap, "\mathbf{C}"] \arrow[ru, Rightarrow, shorten >=20pt, shorten <=20pt, "\mathbf{w}"] & \cdot
\end{tikzcd}
\qquad = \qquad
\begin{tikzcd}[row sep=huge, column sep=huge]
\cdot \arrow[r, "\mathbf{C}"] \arrow[d, swap, "P \times P"] & \cdot \arrow[d, "P"] \\
\cdot \arrow[r, "\mathbf{C}"] \arrow[d, swap, "R \times R"] \arrow[ru, Rightarrow, shorten >=20pt, shorten <=20pt, "\mathbf{x}"] & \cdot \arrow[d, "R"] \\
\cdot \arrow[r, swap, "\mathbf{C}"] \arrow[ru, Rightarrow, shorten >=20pt, shorten <=20pt, "\mathbf{w}"] & \cdot
\end{tikzcd}
\end{equation}
We substitute (\ref{eq2}) back into (\ref{eq1}) to get
\begin{equation*}
\begin{tikzcd}[row sep=huge, column sep=huge]
\cdot \arrow[r, "\mathbf{C}"] \arrow[d, swap, "P \times P"] \arrow[dd, bend right=100, swap, "\mathrm{id}"{name=A}] & \cdot \arrow[d, "P"] \arrow[dd, bend left=100, "\mathrm{id}"{name=B}] \\
\cdot \arrow[r, "\mathbf{C}"] \arrow[d, swap, "R \times R"] \arrow[ru, Rightarrow, shorten >=20pt, shorten <=20pt, "\mathbf{x}"] & \cdot \arrow[d, "R"] \\
\cdot \arrow[r, swap, "\mathbf{C}"] \arrow[ru, Rightarrow, shorten >=20pt, shorten <=20pt, "\mathbf{w}"] & \cdot
\arrow[from=A, to=2-1, Rightarrow, shorten >=8pt, shorten <=8pt, "\mathbf{z} \times \mathbf{z}"] \arrow[from=2-2, to=B, Rightarrow, shorten >=8pt, shorten <=8pt, "\mathbf{z}^{-1}"]
\end{tikzcd}
\qquad = \qquad
\begin{tikzcd}[row sep=huge, column sep=huge]
\cdot \arrow[r, "\mathbf{C}"] \arrow[d, swap, "\mathrm{id}"] & \cdot \arrow[d, "\mathrm{id}"] \\
\cdot \arrow[r, swap, "\mathbf{C}"] \arrow[ru, Rightarrow, shorten >=20pt, shorten <=20pt, "\mathrm{id}"] & \cdot
\end{tikzcd}
\end{equation*}
by Theorem \ref{fordiagthm}.
\end{proof}

\begin{lem} \label{lem4}
The pullbacks of fibrations along any other morphism exist. Furthermore, the resulting morphism can be taken strict.
\end{lem}

\begin{proof}
Given two morphisms $(F, \phi): \mathcal{B} \longrightarrow \mathcal{C}$ and $(G, \gamma): \mathcal{D} \longrightarrow \mathcal{C}$, with $F$ a fibration, we construct a square
\begin{equation} \label{pullb}
\begin{tikzcd}[row sep=huge, column sep=huge]
\mathcal{A} \arrow[r, "{(R, \rho)}"] \arrow[d, swap, "{(P, \pi)}"] & \mathcal{B} \arrow[d, "{(F, \phi)}"] \\
\mathcal{D} \arrow[r, swap, "{(G, \gamma)}"] & \mathcal{C}
\end{tikzcd}
\end{equation}
and demonstrate its universal property. The set of 0-cells $\mathcal{A}_0$, equipped with functions $R : \mathcal{A}_0 \longrightarrow \mathcal{B}_0$ and $P : \mathcal{A}_0 \longrightarrow \mathcal{D}_0$, is given by the pullback square (of sets!)
\begin{equation*}
\begin{tikzcd}[row sep=huge, column sep=huge]
\mathcal{A}_0 \arrow[r, "R"] \arrow[d, swap, "P"] & \mathcal{B}_0 \arrow[d, "F"] \\
\mathcal{D}_0 \arrow[r, swap, "G"] & \mathcal{C}_0
\end{tikzcd}
\end{equation*}
To cut back clutter, we write $PA = D$, $RA = B$ and $FB = GD = C$ for $A$ in $\mathcal{A}_{0}$. Given a pair of 0-cells $A$, $A'$ of $\mathcal{A}$, the groupoid $\mathcal{A}(A, A')$, equipped with functors $P_{A, A'} : \mathcal{A}(A, A') \longrightarrow \mathcal{D}(D, D')$ and $R_{A, A'} : \mathcal{A}(A, A') \longrightarrow \mathcal{B}(B, B')$ is given by the pullback square (of groupoids!)
\begin{equation*}
\begin{tikzcd}[row sep=huge, column sep=huge]
\mathcal{A}(A, A') \arrow[r, "R_{A,A'}"] \arrow[d, swap, "P_{A, A'}"] & \mathcal{B}(B, B') \arrow[d, "F_{B,B'}"] \\
\mathcal{D}(D, D') \arrow[r, swap, "G_{D, D'}"] & \mathcal{C}(C, C')
\end{tikzcd}
\end{equation*}
We will now provide the functor $\mathbf{C}_{A, A', A''} : \mathcal{A}(A',A'') \times \mathcal{A}(A,A') \longrightarrow \mathcal{A}(A,A'')$ for a given triple of 0-cells $A$, $A'$, $A''$. Consider the following square:
\begin{equation} \label{cdef}
\begin{tikzcd}[row sep=huge, column sep=huge]
\mathcal{A}(A',A'') \times \mathcal{A}(A,A') \arrow[r, dashed, bend left=50, "\exists H"{name=H}] \arrow[r, "\mathbf{C} \circ (R \times R)"{name=D}] \arrow[d, swap, "\mathbf{C} \circ (P \times P)"] & \mathcal{B}(B',B'') \times \mathcal{B}(B,B') \arrow[d, "F"] \\
\mathcal{D}(D,D'') \arrow[r, swap, "G"] \arrow[ru, Rightarrow, shorten >=40pt, shorten <=40pt, "\phi R \circ (\gamma P)^{-1}"] & \mathcal{C}(C,C'')
\arrow[from=H, to =D, swap, Rightarrow, dashed, shorten >=5pt, shorten <=5pt, "\exists \alpha"]
\end{tikzcd}
\end{equation}
The calculation
\begin{equation*}
G \circ \mathbf{C} \circ (P \times P) \xRightarrow{(\gamma P)^{-1}} \mathbf{C} \circ (G \times G) \circ (P \times P) = \mathbf{C} \circ (F \times F) \circ (R \times R) \xRightarrow{\; \phi R \;} F \circ \mathbf{C} \circ (R \times R)
\end{equation*}
shows that (\ref{cdef}) indeed commutes up to the natural isomorphism $\phi R \circ (\gamma P)^{-1}$. By Lemma \ref{lem3} there exists a functor $H (= H_{A, A', A''})$ which makes the square commute, along with a natural isomorphism 
\begin{equation*}
\alpha (= \alpha_{A, A', A''}) : H \Longrightarrow \mathbf{C} \circ (R \times R)
\end{equation*}
(both indicated by dashed arrows), such that $F \alpha = \phi R \circ (\gamma P)^{-1}$. By the universal property of $\mathcal{A}(A,A'')$, this commuting square (\ref{cdef}) gives rise to the functor we are looking for
\begin{equation*}
\mathbf{C}_{A, A', A''} = \langle \mathbf{C}_{D, D', D''} \circ (P_{A', A''} \times P_{A, A'}) , H_{A, A', A''} \rangle .
\end{equation*}

We finish the definition of $(P, \pi)$ and $(R, \rho)$ by setting
\begin{equation*}
\pi_{A, A', A''} = \mathrm{id} : \mathbf{C}_{D, D', D''} \circ (P_{A', A''} \times P_{A, A'}) \Longrightarrow P_{A, A''} \circ \mathbf{C}_{A, A', A''}
\end{equation*}
and
\begin{equation*}
\rho_{A, A', A''} = \alpha_{A, A', A''}^{-1} : \mathbf{C}_{B, B', B''} \circ (R_{A', A''} \times R_{A, A'}) \Longrightarrow R_{A, A''} \circ \mathbf{C}_{A, A', A''}.
\end{equation*}
The calculations
\begin{equation*}
(F, \phi) \circ (R, \rho) = (F \circ R, F \rho \circ \phi R) = (F \circ R, F \alpha^{-1} \circ \phi R) = (F \circ R, (\phi R \circ ( \gamma P)^{-1})^{-1} \circ \phi R) = (F \circ R, \gamma P)
\end{equation*}
and
\begin{equation*}
(G, \gamma) \circ (P, \pi) = (G \circ P, G \pi \circ \gamma P) = (G \circ P, \gamma P)
\end{equation*}
show that (\ref{pullb}) commutes.

The definition of $\mathcal{A}$ is finished by letting
\begin{equation*}
\mathbf{a}_{A, A', A'', A'''} : \mathbf{C}_{A, A', A'''} \circ ( \mathbf{C}_{A', A'', A'''} \times \mathrm{id} ) \Longrightarrow \mathbf{C}_{A, A'', A'''} \circ ( \mathrm{id} \times \mathbf{C}_{A, A', A'''} )
\end{equation*}
be the unique natural isomorphism such that for any combination 
\begin{equation*}
A \overset{a}{\longrightarrow} A' \overset{a'}{\longrightarrow} A'' \overset{a''}{\longrightarrow} A''' 
\end{equation*}
of composable 1-cells the diagrams 
\begin{equation*}
\begin{tikzcd}[row sep=huge, column sep=huge]
(Pa'' * Pa') * Pa \arrow[r, "\pi * \mathrm{id}"] \arrow[d, swap, "\mathbf{a}"] & P(a'' * a') * Pa \arrow[r, "\pi"] & P((a'' * a') * a) \arrow[d, dashed, "P \mathbf{a}"] \\
Pa'' * (Pa' * Pa) \arrow[r, swap, "\mathrm{id} * \pi"] & Pa'' * P(a' * a) \arrow[r, swap, "\pi"] & P(a'' * (a' * a))
\end{tikzcd}
\end{equation*}
and
\begin{equation*}
\begin{tikzcd}[row sep=huge, column sep=huge]
(Ra'' * Ra') * Ra \arrow[r, "\rho * \mathrm{id}"] \arrow[d, swap, "\mathbf{a}"] & R(a'' * a') * Ra \arrow[r, "\rho"] & R((a'' * a') * a) \arrow[d, dashed, "R \mathbf{a}"] \\
Ra'' * (Ra' * Ra) \arrow[r, swap, "\mathrm{id} * \rho"] & Ra'' * R(a' * a) \arrow[r, swap, "\rho"] & R(a'' * (a' * a))
\end{tikzcd}
\end{equation*}
commute. (The dashed arrows mark the two projections of $\mathbf{a}_{A, A', A'', A'''}$.) In other words, we force the diagram (\ref{coh4}) to commute.

To show that $\mathcal{A}$ is a bigroupoid, we must verify that the diagrams (\ref{coh1}), (\ref{coh2}) and (\ref{coh3}) commute in $\mathcal{A}$. Since a diagram in $\mathcal{A}$ commutes if and only if the projections of this diagram under $P$ and $R$ commute in $\mathcal{D}$ and $\mathcal{B}$ respectively, this follows from $\textbf{(1)} \Longrightarrow \textbf{(2)}$ of Lemma \ref{lem2}.

Lastly, we demonstrate that our square has the desired universal property:
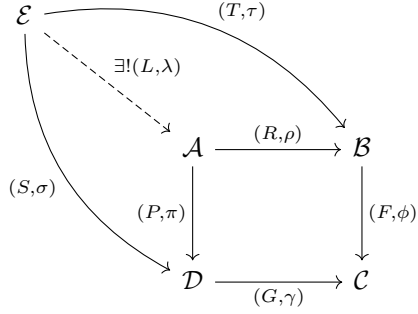
\begin{equation*}
\begin{tikzcd}[row sep=huge, column sep=huge]
\mathcal{E} \arrow[dr, dashed, "{\exists ! (L, \lambda)}"] \arrow[ddr, bend right, swap, "{(S, \sigma)}"] \arrow[drr, bend left, "{(T, \tau)}"]& & \\
& \mathcal{A} \arrow[r, "{(R, \rho)}"] \arrow[d, swap, "{(P, \pi)}"] & \mathcal{B} \arrow[d, "{(F, \phi)}"] \\
& \mathcal{D} \arrow[r, swap, "{(G, \gamma)}"] & \mathcal{C}
\end{tikzcd}
\end{equation*}
It is not difficult to check that there exists a unique incoherent morphism $(L, \lambda):  \mathcal{E} \longrightarrow \mathcal{A}$ satisfying
\begin{equation*}
(S, \sigma) = (P, \pi) \circ (L, \lambda) = (P \circ L, P \lambda \circ \pi L) \qquad \text{and} \qquad (T, \tau) = (R, \rho) \circ (L, \lambda) = (R \circ L, R \lambda \circ \rho L),
\end{equation*}
namely
\begin{align*}
L & = \langle S, T \rangle : \mathcal{E}_{0} \longrightarrow \mathcal{A}_{0} \\
L_{E, E'} & = \langle S_{E, E'}, T_{E, E'} \rangle : \mathcal{E}(E, E') \longrightarrow \mathcal{A}(LE, LE') \\
\lambda & = \langle \sigma \circ (\pi L)^{-1}, \tau \circ (\rho L)^{-1} \rangle.
\end{align*}
To show that $(L, \lambda)$ is a morphism, we must verify that the diagrams (\ref{coh4}) and (\ref{coh5}) commute for $\lambda$. Again, it suffices that the projections of these diagrams under $P$ and $R$ commute in $\mathcal{D}$ and $\mathcal{B}$. But this follows directly from $\textbf{(1)} + \textbf{(3)} \Longrightarrow \textbf{(2)}$ of Lemma \ref{lem1}.
\end{proof}

\begin{lem} \label{lem11}
\hfill
\begin{description}
\item[(1)] Fibrations are closed under composition.
\item[(2)] Every isomorphism is a fibration.
\item[(3)] Fibrations are closed under pullback.
\end{description}
\end{lem}

\begin{proof}
Straightforward. By \textbf{(1)} and \textbf{(2)}, it suffices to check \textbf{(3)} for the explicit construction made in Lemma \ref{lem4}.
\end{proof}

\begin{lem} \label{lem9}
Let $(F, \phi) : \mathcal{A} \longrightarrow \mathcal{C}$ be a trivial cofibration. Then there exists a factorization 
\begin{equation*}
\mathcal{A} \xrightarrow{(G, \gamma)} \mathcal{B} \xrightarrow{(H, \mathrm{id})} \mathcal{C}
\end{equation*}
of $F$, where $G$ is a trivial cofibration which is surjective on 0-cells and $H$ is a strict trivial cofibration which is also a local isomorphism.
\end{lem}

\begin{proof}
Let $\mathcal{B}$ be the sub-bigroupoid of $\mathcal{C}$ consisting of the 0-cells in the image of $F$ with all 1- and 2-cells of $\mathcal{C}$ between them. One easily verifies that the evident morphisms $(G, \gamma) : \mathcal{A} \longrightarrow \mathcal{B}$ and $(H, \mathrm{id}) : \mathcal{B} \longrightarrow \mathcal{C}$ have the desired properties.
\end{proof}

\begin{lem}
The trivial cofibrations have the left lifting property with respect to the fibrations.
\end{lem}

\begin{proof}
Let the lifting problem
\begin{equation} \label{lift1}
\begin{tikzcd}[row sep=huge, column sep=huge]
\mathcal{A} \arrow[r, "{(F, \phi)}"] \arrow[d, swap, "{(K, \kappa)}"] & \mathcal{B} \arrow[d, "{(G, \gamma)}"] \\
\mathcal{D} \arrow[r, swap, "{(H, \eta)}"] \arrow[ru, dashed, "?"] & \mathcal{C}
\end{tikzcd}
\end{equation}
be given, in which $K$ is a trivial cofibration and $G$ is a fibration.

Consider the pullback $\mathcal{E}$, of $G$ along $H$, and apply its universal property to obtain
\begin{equation} \label{liftpull}
\begin{tikzcd}[row sep=huge, column sep=huge]
\mathcal{A} \arrow[rr, bend left, "{(F, \phi)}"] \arrow[r, dashed, swap, "\exists !"] \arrow[d, swap, "{(K, \kappa)}"] & \mathcal{E} \arrow[r] \arrow[d, swap, "{(G', \mathrm{id})}"] & \mathcal{B} \arrow[d, "{(G, \gamma)}"] \\
\mathcal{D} \arrow[r, swap, "\mathrm{id}"] & \mathcal{D} \arrow[r, swap, "{(H, \eta)}"] & \mathcal{C}
\end{tikzcd}
\end{equation}
Note that this pullback exists and yields a strict fibration $G'$ due to Lemma \ref{lem4} and Lemma \ref{lem11}. The observation that a diagonal filler for the left square in (\ref{liftpull}) results in a filler for the original square (\ref{lift1}) establishes that we may assume that (\ref{lift1}) is of the form
\begin{equation} \label{lift2}
\begin{tikzcd}[row sep=huge, column sep=huge]
\mathcal{A} \arrow[r, "{(F, \phi)}"] \arrow[d, swap, "{(K, \kappa)}"] & \mathcal{B} \arrow[d, "{(G, \mathrm{id})}"] \\
\mathcal{C} \arrow[r, swap, "\mathrm{id}"] & \mathcal{C}
\end{tikzcd}
\end{equation}

Factorize $(K, \kappa)$ into $(T, \mathrm{id}) \circ (S, \sigma)$, using Lemma \ref{lem9}. Substituting this into (\ref{lift2}) yields the square
\begin{equation*}
\begin{tikzcd}[row sep=huge, column sep=huge]
\mathcal{A} \arrow[r, "{(F, \phi)}"] \arrow[d, swap, "{(S, \sigma)}"] & \mathcal{B} \arrow[d, "{(G, \mathrm{id})}"] \\
\mathcal{D} \arrow[r, swap, "{(T, \mathrm{id})}"] \arrow[ru, dashed, "{\exists (L, \lambda)}"] & \mathcal{C}
\end{tikzcd}
\end{equation*}
for which the indicated lift $L$ exists by virtue of Lemma \ref{surlift}. Lemma \ref{isolift}, in turn, provides a lift $M$ for the square
\begin{equation*}
\begin{tikzcd}[row sep=huge, column sep=huge]
\mathcal{D} \arrow[r, "{(L, \lambda)}"] \arrow[d, swap, "{(T, \mathrm{id})}"] & \mathcal{B} \arrow[d, "{(G, \mathrm{id})}"] \\
\mathcal{C} \arrow[r, swap, "\mathrm{id}"] \arrow[ru, dashed, "{\exists (M, \mu)}"] & \mathcal{C}
\end{tikzcd}
\end{equation*}
as shown. But then $M$ is a diagonal filler for (\ref{lift2}).
\end{proof}

\subsection{Factorization}

\begin{dfn}
A \emph{path object} on a bigroupoid $\mathcal{B}$ is a factorisation of the diagonal $\Delta : \mathcal{B} \longrightarrow \mathcal{B} \times \mathcal{B}$ as a weak equivalence $R : \mathcal{B} \longrightarrow \mathcal{PB}$ followed by a fibration $\langle S, T \rangle : \mathcal{PB} \longrightarrow \mathcal{B} \times \mathcal{B}$.
\end{dfn}

The construction for path objects that we give below is basically the same as the one given in \cite{MR2138540} for bicategories. 

\begin{lem} \label{lem7}
Every bigroupoid has a path object.
\end{lem}

\begin{proof}
Let $\mathcal{B}$ be a bigroupoid. We construct a path object $\mathcal{PB}$ for $\mathcal{B}$. By virtue of Theorem \ref{fordiagthm}, we allow ourselves to write as if $\mathcal{B}$ were a strict bigroupoid. The set of 0-cells of $\mathcal{PB}$ is the set of all 1-cells of $\mathcal{B}$. Given a pair of 0-cells $a : A \longrightarrow A'$, $b : B \longrightarrow B'$ in $\mathcal{PB}$, a 1-cell $a \longrightarrow b$ is a triple $(f, \phi, f')$, with $f : A \longrightarrow B$, $f' : A' \longrightarrow B'$ and $\phi : f' * a \longrightarrow b * f$. We can visualize such a 1-cell of $\mathcal{PB}$ as a square of 1-cells in $\mathcal{B}$, which commutes up to a 2-cell:
\begin{equation*}
\begin{tikzcd}[row sep=huge, column sep=huge]
A \arrow[r, "f"] \arrow[d, swap, "a"] & B \arrow[d, "b"] \\
A' \arrow[r, swap, "f'"]  \arrow[ru, Rightarrow, shorten >=20pt, shorten <=20pt, "\phi"] & B'
\end{tikzcd}
\end{equation*}
A 2-cell from $(f, \phi, f')$ to $(g, \psi, g')$ is a pair $(\alpha, \alpha')$ of 2-cells $\alpha : f \longrightarrow g$, $\alpha' : f' \longrightarrow g'$ in $\mathcal{B}$, such that the diagram
\begin{equation*}
\begin{tikzcd}[row sep=huge, column sep=huge]
f' a \arrow[r, "\phi"] \arrow[d, swap, "\alpha' * \mathrm{id}"] & b f \arrow[d, "\mathrm{id} * \alpha"] \\
g' a \arrow[r, swap, "\psi"] & b g
\end{tikzcd}
\end{equation*}
commutes. One easily checks that $\mathcal{PB}(a, b)$, defined in this way, forms a groupoid.

Next, we define the functor $\mathbf{C}_{a, b, c} : \mathcal{PB}(b, c) \times \mathcal{PB}(a, b) \longrightarrow \mathcal{PB}(a, c)$. Given two 1-cells $(f, \phi, f') : a \longrightarrow b$ and $(g, \psi, g') : b \longrightarrow c$, we define
\begin{equation*}
(g, \psi, g') * (f, \phi, f') = (g * f, \psi * \phi, g'* f').
\end{equation*}
The composition $\psi * \phi$ makes sense, because we are willfully ignorant about associativity issues. Given four 1-cells 
\begin{equation*}
(f_{1}, \phi_{1}, f_{1}'), (f_{2}, \phi_{2}, f_{2}') : a \longrightarrow b \qquad \text{and} \qquad (g_{1}, \psi_{1}, g_{1}'), (g_{2}, \psi_{2}, g_{2}') : b \longrightarrow c
\end{equation*}
and 2-cells
\begin{equation*}
(\alpha, \alpha') : (f_{1}, \phi_{1}, f_{1}') \longrightarrow (f_{2}, \phi_{2}, f_{2}') \qquad \text{and} \qquad (\beta, \beta') : (g_{1}, \psi_{1}, _{1}') \longrightarrow (g_{2}, \psi_{2}, g_{2}')
\end{equation*}
between them, we define
\begin{equation*}
(\beta, \beta') * (\alpha, \alpha') = (\beta * \alpha, \beta' * \alpha').
\end{equation*}
The commutative diagram
\begin{equation*}
\begin{tikzcd}[row sep=huge, column sep=huge]
g_{1}' f_{1}' a \arrow[r, "\mathrm{id} * \phi_{1}"] \arrow[d, swap, "\beta' * \alpha' * \mathrm{id}"] & g_{1}' b f_{1} \arrow[r, "\psi_{1} * \mathrm{id}"] \arrow[d, swap, "\beta' * \mathrm{id} * \alpha"] & c g_{1} f_{1} \arrow[d, "\mathrm{id} * \beta * \alpha"] \\
g_{2}' f_{2}' a \arrow[r, swap, "\mathrm{id} * \phi_{2}"] & g_{2}' b f_{2} \arrow[r, swap, "\psi_{2} * \mathrm{id}"] & c g_{2} f_{2}
\end{tikzcd}
\end{equation*}
confirms that $(\beta * \alpha, \beta' * \alpha')$ is in fact a 2-cell.

Next, for any four 0-cells $a: A \longrightarrow A'$, $b: B \longrightarrow B'$, $c: C \longrightarrow C'$, $d: D \longrightarrow D'$ in $\mathcal{PB}$, we define the natural isomorphism $\mathbf{a}_{a, b, c, d}$. Given 1-cells $(f, \phi, f') : a \longrightarrow b$, $(g, \psi, g') : b \longrightarrow c$ and $(h, \theta, h') : c \longrightarrow d$, we take 
\begin{equation*}
(\mathbf{a}_{a, b, c, d})_{((h, \theta, h'), (g, \psi, g'), (f, \phi, f'))} = ((\mathbf{a}_{A, B, C, D})_{(h, g, f)}, (\mathbf{a}_{A', B', C', D'})_{(h', g', f')}).
\end{equation*}
In order for this to be a genuine 2-cell, the diagram 
\begin{equation} \label{pathsq}
\begin{tikzcd}[row sep=huge, column sep=huge]
((h' g') f') a \arrow[r, "(\theta * \psi) * \phi"] \arrow[d, swap, "\mathbf{a} * \mathrm{id}"] & d ((h g) f) \arrow[d, "\mathrm{id} * \mathbf{a}"] \\
(h' (g' f')) a \arrow[r, swap, "\theta * (\psi * \phi)"] & d (h (g f))
\end{tikzcd}
\end{equation}
must commute. Since we may calculate as if $\mathcal{B}$ were strict, we can remove all brackets appearing in (\ref{pathsq}) and set $\mathbf{a} = \mathrm{id}$, resulting in a square that trivially commutes. The diagrams (\ref{coh1}), (\ref{coh2}) and (\ref{coh3}) commute simply because they commute componentwise, hence $\mathcal{PB}$ is a bigroupoid.

The diagonal $\Delta : \mathcal{B} \longrightarrow \mathcal{B} \times \mathcal{B}$ now factors trough $\mathcal{PB}$ as the strict morphism $R : \mathcal{B} \longrightarrow \mathcal{PB}$, which
\begin{itemize}
\item{sends a 0-cell $A$ to $1_{A} : A \longrightarrow A$,}
\item{sends a 1-cell $f : A \longrightarrow B$ to $(f, \phi, f)$, with $\phi : f * 1_{A} \longrightarrow 1_{B} * f$ canonical}
\item{and sends a 2-cell $\alpha : f \longrightarrow g$ to $(\alpha, \alpha)$,}
\end{itemize}
followed by the strict morphism $\langle S, T \rangle : \mathcal{B} \longrightarrow \mathcal{PB}$, which
\begin{itemize}
\item{sends a 0-cell $a : A \longrightarrow A'$ to $(A, A')$,}
\item{sends a 1-cell $(f, \phi, f')$ to $(f, f')$}
\item{and sends a 2-cell $(\alpha, \alpha')$ to $(\alpha, \alpha')$.}
\end{itemize}
We leave it to the reader to verify that $R$ and $\langle S, T \rangle$ satisfy the necessary conditions.
\end{proof}

The following Lemma collects some miscellaneous results, to be used in Lemma \ref{lem10}.

\begin{lem}
\hfill
\begin{description}
\item[(1)] Trivial fibrations are closed under pullback.
\item[(2)] For every bigroupoid $\mathcal{B}$, the unique morphism $\mathcal{B} \longrightarrow 1$ is a fibration.
\item[(3)] Every split monomorphism is a cofibration.
\end{description}
\end{lem}

\begin{proof}
Straightforward. For \textbf{(1)}, note that the trivial fibrations form the right class of a weak factorization system.
\end{proof}

The following argument is originally due to Brown \cite{MR0341469}.

\begin{lem} \label{lem10}
Let $(F, \phi) : \mathcal{A} \longrightarrow \mathcal{C}$ be a morphism of bigroupoids. Then there exists a factorization 
\begin{equation*}
\mathcal{A} \xrightarrow{(G, \psi)} \mathcal{B} \xrightarrow{(H, \eta)} \mathcal{C}
\end{equation*}
of $F$, where $G$ is a trivial cofibration and $H$ is a fibration.
\end{lem}

\begin{proof}
Since the unique morphism $\mathcal{C} \longrightarrow 1$ is a fibration and fibrations are closed under pullback, the two projections $\mathcal{C} \times \mathcal{C} \longrightarrow \mathcal{C}$ are fibrations as well. Since fibrations are closed under composition, it follows that $S : \mathcal{PC} \longrightarrow \mathcal{C}$ (with $\begin{tikzcd}[column sep=large]
\mathcal{C} \arrow[r, "R" description] &[-15pt] \mathcal{PC} \arrow[r, "{\langle S, T \rangle}" description]  & \mathcal{C} \times \mathcal{C}
\end{tikzcd}$ as in Lemma \ref{lem7}) is a fibration. We can therefore take the pullback of $S$ along $F$ and apply its universal property, as depicted below 
\begin{equation*}
\begin{tikzcd}[row sep=huge, column sep=huge]
\mathcal{A} \arrow[dr, dashed, "\exists ! G"] \arrow[ddr, bend right, swap, "\mathrm{id}"] \arrow[drr, bend left, "R \circ F"] & & \\
& \mathcal{B} \arrow[r, "Q"] \arrow[d, swap, "P"] & \mathcal{PC} \arrow[d, "S"] \\
& \mathcal{A} \arrow[r, swap, "F"] & \mathcal{C}
\end{tikzcd}
\end{equation*}
Since $S \circ R = \mathrm{id}$ and $R$ is a weak equivalence, 2-out-of-3 implies that $S$ is a weak equivalence and hence a trivial fibration. These are stable under pullback, so $P$ is a trivial fibration as well. The equality $P \circ G = \mathrm{id}$ then shows that $G$ is a weak equivalence, by 2-out-of-3. It also shows that $G$ is a split monomorphism and therefore a (trivial) cofibration. Defining $H = T \circ Q$ yields a factorization $F = H \circ G$. The square
\begin{equation*}
\begin{tikzcd}[row sep=huge, column sep=huge]
\mathcal{B} \arrow[r, "Q"] \arrow[d, swap, "{\langle P, H \rangle}"] & \mathcal{PC} \arrow[d, "{\langle S, T \rangle}"] \\
\mathcal{A} \times \mathcal{C} \arrow[r, swap, "F \times \mathrm{id}"] & \mathcal{C} \times \mathcal{C}
\end{tikzcd}
\end{equation*}
exhibits $\langle P, H \rangle$ as a pullback (by the pullback Lemma) of the fibration $\langle S, T \rangle$, which implies that $H$ is a fibration as well.
\end{proof}

With this, Proposition \ref{prop} is proven, which also finishes the proof of Theorem \ref{mainthm}.

\begin{rmk}
Note that the only place where we seem to make essential use of the fact that we are working with \textit{bigroupoids} and not \textit{bicategories} is Lemma \ref{isolift}. It is quite possible that this may be adapted somehow, resulting in a model structure on the category of (small) bicategories and pseudofunctors.
\end{rmk}

\appendix

\section{Coherence for AU-bigroupoids}

In this section we prove a coherence theorem for `AU-bigroupoids' (Definition \ref{audef}). This is an intermediate step in the proof a coherence theorem for bigroupoids. Our approach closely follows that of \cite{MR723395}, which is in turn based on \cite{MR641327}.

\begin{dfn} \label{audef}
An \textit{associative unital bigroupoid} or \textit{AU-bigroupoid} is a bigroupoid in which the natural isomorphisms $\mathbf{a}$, $\mathbf{l}$ and $\mathbf{r}$ are identities.
\end{dfn}

\begin{rmk} \label{idrmk}
Since identity 1-cells are strict in an AU-bigroupoid, the 2-cells $\alpha : f \longrightarrow g$ and $\alpha * \mathrm{id} : f * 1 \longrightarrow g * 1$ are identical. If it is not clear why a certain diagram commutes, it may sometimes prove helpful to introduce such a `missing' 1.
\end{rmk}

The following Lemma is a result of the fact that in an adjoint equivalence, the two triangle identities imply one another.

\begin{lem} \label{triidlem}
Let $\mathcal{B}$ be a AU-bigroupoid. Then for every 1-cell $f$ of $\mathcal{B}$ the following two diagrams commute
\begin{equation} \label{triid}
\begin{tikzcd}[row sep=huge, column sep=huge]
f \arrow[r, "\mathbf{i} * \mathrm{id}"] \arrow[dr, swap, "\mathrm{id}"] & f f^{*} f \arrow[d, "\mathrm{id} * \mathbf{e}"] & f^{*} \arrow[r, "\mathrm{id} * \mathbf{i}"] \arrow[dr, swap, "\mathrm{id}"] & f^{*} f f^{*} \arrow[d, "\mathbf{e} * \mathrm{id}"] \\
& f & & f^{*}
\end{tikzcd}
\end{equation}
\end{lem}

\begin{proof}
Commutativity of the left triangle of (\ref{triid}) is just the coherence requirement (\ref{coh3}). For the triangle on the right, consider the diagram
\begin{equation*}
\begin{tikzcd}[row sep=huge, column sep=huge]
f^{*} \arrow[r, "\mathrm{id} * \mathbf{i}"] \arrow[d, swap, "\mathrm{id} * \mathbf{i}"] & f^{*} f f^{*} \arrow[d, swap, "\mathrm{id} * \mathbf{i} * \mathrm{id}"] \arrow[dr, "\mathrm{id}"] & \\
f^{*} f f^{*} \arrow[r, "\mathrm{id} * \mathbf{i}"] \arrow[d, swap, "\mathbf{e} * \mathrm{id}"] & f^{*} f f^{*} f f^{*} \arrow[r, "\mathrm{id} * \mathbf{e} * \mathrm{id}"] & f^{*} f f^{*} \arrow[d, "\mathbf{e} * \mathrm{id}"] \\
f^{*} \arrow[r, swap, "\mathrm{id} * \mathbf{i}"] & f^{*} f f^{*} \arrow[r, swap, "\mathbf{e} * \mathrm{id}"] & f^{*}
\end{tikzcd}
\end{equation*}
The top left square of this diagram commutes, as both traversals give $\mathrm{id} * \mathbf{i} * \mathbf{i}$ (using Remark \ref{idrmk}); its top right triangle commutes by the left triangle of (\ref{triid}); and the bottom rectangle commutes by naturality of $\mathbf{e}$. The commutativity of the perimeter of this diagram implies that the composition $(\mathbf{e} * \mathrm{id}) \circ (\mathrm{id} * \mathbf{i})$, of its bottom two components must be the identity.
\end{proof}

The next Lemma is due to the fact that a conjugate pair of natural transformations (i.e. a morphism of adjoints) is already uniquely determined by one of its two components.

\begin{lem} \label{invlem}
Let $\alpha : f \longrightarrow g$ be a 2-cell in a AU-bigroupoid. Then the 2-cell $\alpha^{*} : f^{*} \longrightarrow g^{*}$ is equal to the composite
\begin{equation*}
f^{*} \xrightarrow{\mathbf{e}^{-1} * \mathrm{id}} g^{*} g f \xrightarrow{\mathrm{id} * \alpha^{-1} * \mathrm{id}} g^{*} f f^{*} \xrightarrow{\mathrm{id} * \mathbf{i}^{-1}} g^{*}.
\end{equation*}
\end{lem}

\begin{proof}
Consider the diagram
\begin{equation*}
\begin{tikzcd}[row sep=huge, column sep=huge]
f^{*} \arrow[r, "\mathrm{id} * \mathbf{i}"] \arrow[d, swap, "\alpha^{*}"] & f^{*} f f^{*} \arrow[r, "\mathrm{id}"] \arrow[d, swap, "\alpha^{*} * \mathrm{id}"] & f^{*} f f^{*} \arrow[r, "\mathbf{e} * \mathrm{id}"] \arrow[d, swap, "\alpha^{*} * \alpha * \mathrm{id}"] & f^{*} \arrow[d, "\mathrm{id}"] \\
g^{*} \arrow[r, swap, "\mathrm{id} * \mathbf{i}"] & g^{*} f f^{*} \arrow[r, swap, "\mathrm{id} * \alpha * \mathrm{id}"] & g^{*} g f^{*} \arrow[r, swap, "\mathbf{e} * \mathrm{id}"] & f^{*}
\end{tikzcd}
\end{equation*}
It is not difficult to see that the left and middle squares of this diagram commute. Since its rightmost square commutes by naturality of $\mathbf{e}$, the perimeter of the diagram commutes as well. The Lemma now follows by noting that the composition $(\mathbf{e} * \mathrm{id}) \circ \mathrm{id} \circ (\mathrm{id} * \mathbf{i})$, of the top three components of the perimeter is equal to the identity by Lemma \ref{triidlem}.
\end{proof}

\begin{dfn}
Let $\mathcal{B}$ be a AU-bigroupoid. Then for every 1-cell $f$ of $\mathcal{B}$ we define the 2-cell
\begin{equation*}
\mathbf{u}_{f} : f^{**} \longrightarrow f
\end{equation*}
to be the composite
\begin{equation*}
f^{**}  \xrightarrow{\mathrm{id} * \mathbf{e}^{-1}} f^{**} f^{*} f \xrightarrow{\mathbf{e} * \mathrm{id}} f .
\end{equation*}
\end{dfn}

\begin{lem} \label{ulem1}
Let $\mathcal{B}$ be a AU-bigroupoid. Then for every 1-cell $f$ of $\mathcal{B}$ the following two diagrams commute
\begin{equation*}
\begin{tikzcd}[row sep=huge, column sep=huge]
1 \arrow[r, "\mathbf{e}^{-1}"] \arrow[dr, swap, "\mathbf{i}"] & f^{**} f^{*} \arrow[d, "\mathbf{u} * \mathrm{id}"] & 1 \arrow[r, "\mathbf{i}"] \arrow[dr, swap, "\mathbf{e}^{-1}"] & f^{*} f^{**} \arrow[d, "\mathrm{id} * \mathbf{u}"] \\
& f f^{*} & & f^{*} f
\end{tikzcd}
\end{equation*}
\end{lem}

\begin{proof}
We shall only concern ourselves with proving the commutativity of the left triangle. The triangle on the right is susceptible to a similar approach. Consider the diagram
\begin{equation*}
\begin{tikzcd}[row sep=huge, column sep=huge]
1 \arrow[r, "\mathbf{e}^{-1}"] \arrow[d, swap, "\mathbf{i}"] & f^{**} f^{*} \arrow[d, swap, "\mathrm{id} * \mathbf{i}"] \arrow[dr, "\mathrm{id}"] & \\
f f^{*} \arrow[r, swap, "\mathbf{e}^{-1} * \mathrm{id}"] & f^{**} f^{*} f f^{*} \arrow[r, swap, "\mathrm{id} * \mathbf{e} * \mathrm{id}"] & f^{**} f^{*}
\end{tikzcd}
\end{equation*}
The left square of this diagram commutes, as both traversals give $\mathbf{e}^{-1} * \mathbf{i}$ (using Remark \ref{idrmk}). The triangle in the right half of the diagram commutes by Lemma \ref{triidlem}. Since the composition, $(\mathrm{id} * \mathbf{e} * \mathrm{id}) \circ (\mathbf{e}^{-1} * \mathrm{id})$, of the bottom two components of the diagram is by definition equal to $\mathbf{u}^{-1} * \mathrm{id}$, we are done.
\end{proof}

\begin{lem} \label{ulem2}
Let $\mathcal{B}$ be a AU-bigroupoid. Then for every 1-cell $f$ of $\mathcal{B}$ the following diagram commutes
\begin{equation*}
\begin{tikzcd}[row sep=huge, column sep=huge]
f^{***} f^{**} \arrow[r, "\mathbf{u} * \mathbf{u}"] \arrow[dr, swap, "\mathbf{e}"] & f^{*} f \arrow[d, "\mathbf{e}"] \\
& 1
\end{tikzcd}
\end{equation*}
\end{lem}

\begin{proof}
This can be read of directly from
\begin{equation*}
\begin{tikzcd}[row sep=huge, column sep=huge]
& 1 \arrow[dl, swap, "\mathbf{e}^{-1}"] \arrow[d, swap, "\mathbf{i}"] \arrow[dr, "\mathbf{e}^{-1}"] & \\
f^{***} f^{**} \arrow[r, swap, "\mathbf{u} * \mathrm{id}"] & f^{*} f^{**} \arrow[r, swap, "\mathrm{id} * \mathbf{u}"] & f^{*} f
\end{tikzcd}
\end{equation*}
which commutes by Lemma \ref{ulem1}.
\end{proof}

\begin{lem} \label{eqlem}
Let $\mathcal{B}$ be an AU-bigroupoid. Let $A, B, C$ and $D$ be 0-cells and let $f : B \longrightarrow C$ be a 1-cell of $\mathcal{B}$. Then the functors $f * - : \mathcal{B}(A, B) \longrightarrow \mathcal{B}(A, C)$ and $- * f : \mathcal{B}(C, D) \longrightarrow \mathcal{B}(B, D)$ are equivalences of categories, with $f^{*} * -$ and $- * f^{*}$ as their respective pseudo inverses.
\end{lem}

\begin{proof}
Trivial. 
\end{proof}

\begin{dfn}
Let $\mathcal{B}$ be a AU-bigroupoid. Then for every pair of composable 1-cells
\begin{equation*}
A \overset{f}{\longrightarrow} B \overset{g}{\longrightarrow} C
\end{equation*}
of $\mathcal{B}$, we define
\begin{equation*}
\mathbf{b}_{f,g} : (gf)^{*} \longrightarrow f^{*} g^{*}
\end{equation*}
to be the unique 2-cell making the diagram
\begin{equation*}
\begin{tikzcd}[row sep=huge, column sep=huge]
(gf)^{*} gf \arrow[r, "\mathbf{b} * \mathrm{id}"] \arrow[d, swap, "\mathbf{e}"] & f^{*} g^{*} g f \arrow[d, "\mathrm{id} * \mathbf{e} * \mathrm{id}"]  \\
1 & f^{*} f \arrow[l, "\mathbf{e}"]
\end{tikzcd}
\end{equation*}
commute. The existence and uniqueness of such a 2-cell follows from Lemma \ref{eqlem}.
\end{dfn}

\begin{dfn}
A \textit{graph} $\mathcal{G}$ consists of a set (of \textit{nodes} or \textit{0-cells}) $\mathcal{G}_{0}$ and associates to every pair $A, B \in \mathcal{G}_{0}$ a set $\mathcal{G}(A, B)$ (of \textit{edges} or \textit{1-cells}). The collection of graphs forms a category, with morphisms $F : \mathcal{G} \longrightarrow \mathcal{G}'$ consisting of a function $F : \mathcal{G}_{0} \longrightarrow \mathcal{G}_{0}'$ and functions $F_{A, B} : \mathcal{G}(A, B) \longrightarrow \mathcal{G}'(FA, FB)$ for every pair $A, B \in \mathcal{G}_{0}$.
\end{dfn}

\begin{rmk}
Note that every bigroupoid $\mathcal{B}$ has an underlying graph, formed by its 0- and 1-cells. In fact, this gives rise to a forgetful functor from bigroupoids to graphs, which has an associated free functor if we only consider strict morphisms between bigroupoids. We will not introduce additional notation for the forgetful functor, but instead trust that it will be clear from the context whenever we regard a bigroupoid as a graph. 
\end{rmk}

\begin{lem}
Given a graph $\mathcal{G}$, the free AU-bigroupoid $\mathcal{F}_{a} \mathcal{G}$ on $\mathcal{G}$ exists. We record its universal property:
\begin{itemize}
\item{There exists an inclusion of graphs (the unit of the adjunction), $I_{a} : \mathcal{G} \longrightarrow \mathcal{F}_{a} \mathcal{G}$, such that:}
\item{Given a AU-bigroupoid $\mathcal{B}$ and a morphism $F : \mathcal{G} \longrightarrow \mathcal{B}$ of graphs, there exists a unique strict morphism of bigroupoids $\widetilde{F} : \mathcal{F}_{a} \mathcal{G} \longrightarrow \mathcal{B}$ such that $F = \widetilde{F}I_{a}$.}
\end{itemize}
\end{lem}

\begin{cns} \label{fracns}
We sketch a construction of $\mathcal{F}_{a} \mathcal{G}$ and leave it to the reader to verify that this object has the required universal property.

The 0-cells of $\mathcal{F}_{a} \mathcal{G}$ are the nodes of $\mathcal{G}$. For every node $A$ of $\mathcal{G}$, we add a new edge $1_{A} : A \longrightarrow A$. We formally close the edges under the operations $- * -$ and $-^{*}$, taking into account the sources and targets in the obvious way. We quotient out by the congruence relation generated by the requirements that $- * -$ is associative and $1$ acts as identity. The 1-cells of $\mathcal{F}_{a} \mathcal{G}$ are the equivalence classes under this quotient. 

For every 1-cell $f$ of $\mathcal{FG}$, we create 2-cells $\mathbf{e}_{f}$, $\mathbf{i}_{f}$, $\mathbf{e}_{f}^{-1}$, $\mathbf{i}_{f}^{-1}$ and $\mathrm{id}_{f}$. We close the 2-cells under the operations $- * -$, $-^{*}$ and $- \circ -$ (whenever these operations make sense). We quotient out by the congruence relation generated by the requirements that $- \circ -$ and $- * -$ are associative; $\mathrm{id}$ acts as identity; $- ^{-1}$ acts as inverse; $- * -$ and $-^{*}$ are functors; $\mathbf{e}$ and $\mathbf{i}$ are natural; and lastly that the coherence law (\ref{coh3}) holds. The 2-cells of $\mathcal{F}_{a} \mathcal{G}$ are the equivalence classes under this quotient.
\end{cns}

In a group, we may write the element $((a^{-1})^{-1} b)^{-1}$ more cleanly as $b^{-1} a^{-1}$. We can do something similar by `rewriting' the 1-cells of $\mathcal{F}_{a} \mathcal{G}$ into isomorphic, but easier to handle 1-cells. This rewriting is done systematically by means of a strict morphism of 2-categories, $R$.

\begin{cns} \label{fcon}
We construct a strict morphism of 2-categories $R : \mathcal{F}_{a} \mathcal{G} \longrightarrow \mathcal{F}_{a} \mathcal{G}$ which is the identity on 0-cells, along with a $\mathcal{G}_{0} \times \mathcal{G}_{0}$-indexed family of natural isomorphisms $\rho : \mathrm{id} \Longrightarrow R$ (with $\rho_{A, B} : \mathrm{id}_{A,B} \Longrightarrow F_{A,B}$).

We let $R$ be the identity on 0-cells. We inductively define the action of $R$ on 1-cells simultaneously with the components of $\rho$, making several case distinctions. To make sure this procedure is well-defined, let us agree to delete any superfluous occurrences of $1$, not appearing as $1^{*}$ in every 1-cell $u$ of $\mathcal{F}_{a} \mathcal{G}$ (e.g. if $u = 1^{**} * (f * 1)^{*}$, we write $1^{**} * f^{*}$ instead).
\begin{itemize}
\item{If $u$ is of the form $f, f^{*}$ or $1$, with $f$ in $\mathcal{G}$, then $Ru = u$ and $\rho_{u}$ is given by
\begin{equation*}
u \xrightarrow{\;\; \mathrm{id} \;\;} u = Ru.
\end{equation*}}
\item{If $u$ is of the form $1^{*}$, then $R 1^{*} = 1$ and $\rho_{u}$ is given by
\begin{equation*}
1^{*} = 1^{*} * 1 \xrightarrow{\;\; \mathbf{e} \;\;} 1 = R1^{*}.
\end{equation*}}
\item{If $u$ is of the form $v^{**}$, then $R v^{**} = R v$ and $\rho_{u}$ is given by
\begin{equation*}
v^{**} \xrightarrow{\;\; \mathbf{u} \;\;} v \xrightarrow{\;\; \rho_{v} \;\;} Rv = Rv^{**}.
\end{equation*}}
\item{If $u$ is of the form $w * v$, then $R(w * v) = Rw * Rv$ and $\rho_{u}$ is given by
\begin{equation*}
w * v \xrightarrow{\;\; \rho_{w} * \rho_{v} \;\;} Rw * Rv = R(w * v).
\end{equation*}
Note that this is well-defined with respect to 1-cells of the form $v_{1} * v_{2} * \cdots * v_{n}$.}
\item{If $u$ is of the form $(w*v)^{*}$, then $R (w*v)^{*} = R v^{*} * Rw^{*}$ and $\rho_{u}$ is given by
\begin{equation*}
(w*v)^{*} \xrightarrow{\;\; \mathbf{b} \;\;} v^{*} * w^{*} \xrightarrow{\;\; \rho_{v^{*}} \rho_{w^{*}} \;\;} R v^{*} * Rw^{*} = R (w*v)^{*}.
\end{equation*}}
\end{itemize}
We define $R$ on a 2-cell $\alpha : u \longrightarrow v$ by requiring that the square
\begin{equation*}
\begin{tikzcd}[row sep=huge, column sep=huge]
u \arrow[r, "\rho_{w}"] \arrow[d, swap, "\alpha"] & Ru \arrow[d, "R \alpha"] \\
v \arrow[r, swap, "\rho_{v}"] & Rv
\end{tikzcd}
\end{equation*}
commutes. One easily verifies that $R$ is a strict morphism of 2-categories.
\end{cns}

\begin{lem} \label{flem}
The strict morphism of 2-categories $R : \mathcal{F}_{a} \mathcal{G} \longrightarrow \mathcal{F}_{a} \mathcal{G}$ of Construction \ref{fcon} enjoys the following properties:
\begin{description}
\item[(1)] If $u$ is a 1-cell of $\mathcal{F}_{a} \mathcal{G}$, then $Ru = u$ if and only if $u$ is a composition of 1-cells of the form $f, f^{*}$ and $1$, with $f$ in $\mathcal{G}$.
\item[(2)] If $u$ is a 1-cell of $\mathcal{F}_{a} \mathcal{G}$ and $Ru = u$, then $\rho : u \longrightarrow Ru$ is the identity.
\item[(3)] $R$ is an idempotent biequivalence.
\item[(4)] All 2-cells of the form $R \mathbf{u}$ and $R \mathbf{b}$ are identities.
\end{description}
\end{lem}

\begin{proof}
A straightforward check.
\end{proof}

\begin{dfn}
A 2-cell of $\mathcal{F}_{a} \mathcal{G}$ is called \textit{simple} if it can be written as $\mathrm{id} * \mathbf{e}_{f} * \mathrm{id}$, $\mathrm{id} * \mathbf{i}_{f} * \mathrm{id}$, $\mathrm{id} * \mathbf{e}_{f}^{-1} * \mathrm{id}$ or $\mathrm{id} * \mathbf{i}_{f}^{-1} * \mathrm{id}$, with $f$ in $\mathcal{G}$. Note that for example $\mathbf{e}_{f}$ and $\mathbf{i}_{f}$ are included in this definition, using Remark \ref{idrmk}.
\end{dfn}

\begin{lem} \label{elem}
For any 1-cell $u$ of $\mathcal{F}_{a} \mathcal{G}$, the 2-cell $R \mathbf{e}_{u}$ is the identity or can be obtained by (vertically) composing finitely many simple 2-cells. 
\end{lem}

\begin{proof}
We use induction on the number of symbols in $u$, where we uphold the convention on the appearances of $1$, as in Construction \ref{fcon}. Recall that $R \mathbf{e}_{u}$ is defined by the commutative diagram
\begin{equation*}
\begin{tikzcd}[row sep=huge, column sep=huge]
u^{*} * u \arrow[r, "\rho_{u^{*} * u}"] \arrow[dr, swap, "\mathbf{e}_{u}"]  & Ru^{*} * Ru \arrow[d, "R \mathbf{e}_{u}"] \\
& 1
\end{tikzcd}
\end{equation*}
\begin{itemize}
\item{If $u = f$, for some $f$ of $G$, then $\rho_{u^{*} * u} = \mathrm{id}$, so $R \mathbf{e}_{u} = \mathbf{e}_{f}$.}
\item{If $u = f^{*}$, for some $f^{*}$ of $G$, then $\rho_{u^{*} * u} = \mathbf{u}_{f} * \mathrm{id}$. Comparing this with Lemma \ref{ulem1} yields $R \mathbf{e}_{u} = \mathbf{i}_{f}^{-1}$.}
\item{If $u = 1$, then $\rho_{u^{*} * u} = \mathbf{e}_{1}$, so $R \mathbf{e}_{u} = \mathrm{id}$.}
\item{If $u = 1^{*}$, then $\rho_{u^{*} * u} = \mathbf{u}_{1} * \mathbf{e}_{1}$, which means that the outer square of
\begin{equation*}
\begin{tikzcd}[row sep=huge, column sep=huge]
1^{**} * 1^{*} \arrow[r, "\mathbf{u} * \mathrm{id}"] \arrow[d, swap, "\mathbf{e}"]  & 1 * 1^{*} \arrow[d, "\mathrm{id} * \mathbf{e}"] \arrow[dl, swap, "\mathbf{i}^{-1}"] \\
1 & 1 \arrow[l, "R \mathbf{e}"]
\end{tikzcd}
\end{equation*}
commutes. By Lemma \ref{ulem1} upper left triangle commutes as well, which forces the commutativity of the lower right triangle. Comparing this with Lemma \ref{triidlem} yields $R \mathbf{e}_{u} = \mathrm{id}$.}
\item{If $u = v^{**}$, then $\rho_{u^{*} * u} = \mathbf{u}_{v^{*}} * \mathbf{u}_{v}$. Comparing this with Lemma \ref{ulem2} yields $R \mathbf{e}_{u} = \mathbf{e}_{v}$, for which we may apply the induction hypothesis.}
\item{If $u = w * v$, then by definition of $\mathbf{b}_{v, w}$,
\begin{equation*}
\mathbf{e}_{u} = \mathbf{e}_{w} \circ ( \mathrm{id} * \mathbf{e}_{v} * \mathrm{id}) \circ (\mathbf{b}_{v, w} * \mathrm{id}).
\end{equation*}
By strictness of $R$ and part \textbf{(4)} of Lemma \ref{flem}, the application of $R$ to both sides of this equation gives
\begin{equation*}
R \mathbf{e}_{u} = R \mathbf{e}_{w} \circ ( \mathrm{id} * R \mathbf{e}_{v} * \mathrm{id}),
\end{equation*}
which allows us to use the induction hypothesis.}
\item{If $u = (w * v)^{*}$, then by naturality of $\mathbf{e}$,
\begin{equation*}
\mathbf{e}_{u} = \mathbf{e}_{v^{*} * w^{*}} \circ (\mathbf{b}_{v, w}^{*} * \mathbf{b}_{v, w}),
\end{equation*}
which means that
\begin{equation*}
R \mathbf{e}_{u} = R \mathbf{e}_{v^{*} * w^{*}} \circ ( R \mathbf{b}_{v, w}^{*} * \mathrm{id}).
\end{equation*}
Now, by Lemma \ref{invlem},
\begin{equation*}
\mathbf{b}_{v, w}^{*} = (\mathrm{id} * \mathbf{i}_{(w*v)^{*}}) \circ ( \mathrm{id} * \mathbf{b}_{v, w}^{-1} * \mathrm{id}) \circ ( \mathbf{e}_{v^{*} * w^{*}} * \mathrm{id}),
\end{equation*}
so
\begin{equation*}
R \mathbf{b}_{v, w}^{*} = (\mathrm{id} * R \mathbf{i}_{(w*v)^{*}}) \circ ( R \mathbf{e}_{v^{*} * w^{*}} * \mathrm{id}),
\end{equation*}
Lastly, by Lemma \ref{ulem1},
\begin{equation*}
\mathbf{i}_{(w*v)^{*}} = ( \mathrm{id} * \mathbf{u}_{w * v}^{-1} ) \circ \mathbf{e}_{w * v}^{-1},
\end{equation*}
giving
\begin{equation*}
R \mathbf{i}_{(w*v)^{*}} = R \mathbf{e}_{w * v}^{-1}.
\end{equation*}
By combining the above computations, we obtain
\begin{equation*}
R \mathbf{e}_{u} = R \mathbf{e}_{v^{*} * w^{*}} \circ ((( \mathrm{id} * R \mathbf{e}_{w * v}^{-1} ) \circ ( R \mathbf{e}_{v^{*} * w^{*}} * \mathrm{id}) ) * \mathrm{id} ).
\end{equation*}
We can now treat the occurrences of $R \mathbf{e}_{v^{*} * w^{*}}$ as in the previous step, after which we may apply the induction hypothesis.}
\end{itemize}
\end{proof}

\begin{lem} \label{simplem}
Let $u$ and $v$ be 1-cells of $\mathcal{F}_{a} \mathcal{G}$ such that $Ru = u$ and $Rv = v$. Then any 2-cell $\alpha : u \longrightarrow v$ is the identity or can be obtained by (vertically) composing finitely many simple 2-cells.
\end{lem}

\begin{proof}
Using Lemma \ref{invlem}, we start by systematically removing all occurrences of $-^{*}$ appearing in $\alpha$. We can subsequently replace every occurrence of $\mathbf{i}$ by occurrences of $\mathbf{e}$, using Lemma \ref{ulem1}. By Lemma \ref{elem}, the 2-cell $R \alpha$ now has the required property. But $\alpha = R \alpha$, as an immediate consequence of Lemma \ref{flem} \textbf{(2)}.
\end{proof}

\begin{dfn}
Define the \textit{length} of a 1-cell of $\mathcal{F}_{a} \mathcal{G}$ to be the number of edges of $\mathcal{G}$ occurring in it, counted with multiplicity (e.g. $\mathrm{length}( f * (f * 1)^{*} ) = 2$).
\end{dfn}

\begin{dfn}
A 2-cell $\alpha : u \longrightarrow v$ of $\mathcal{F}_{a} \mathcal{G}$ is called a \textit{simple reduction} if it is simple and $\mathrm{length}(v) < \mathrm{length}(u)$. We say that a 2-cell of $\mathcal{F}_{a} \mathcal{G}$ is a \textit{reduction} if it is an identity or it can be obtained by (vertically) composing finitely many simple reductions.
\end{dfn}

The next Lemma shows that we are in a setting in which a `Diamond Lemma' can be applied. For us, 2-cells will take the place of the binary relation in terms of which the classical Diamond Lemma is usually formulated. This does not create any difficulties and the proof will be essentially that of the classical Lemma.

\begin{lem} \label{diam1}
Let $u$ be a 1-cell of $\mathcal{F}_{a} \mathcal{G}$. Then for any two simple reductions $\alpha : u \longrightarrow v$ and $\alpha' : u \longrightarrow v'$, there exist reductions $\beta : v \longrightarrow w$ and $\beta' : v' \longrightarrow w$ completing the commutative `diamond' below
\begin{equation*}
\begin{tikzcd}[row sep=huge, column sep=huge]
u \arrow[r, "\alpha"] \arrow[d, swap, "\alpha'"] & v \arrow[d, dashed, "\beta"] \\
v' \arrow[r, swap, dashed, "\beta'"] & w
\end{tikzcd}
\end{equation*}
\end{lem}

\begin{proof}
The proof is just a matter of making a few case distinctions. In what follows, $x,y,z$ are arbitrary 1-cells of $\mathcal{F}_{a} \mathcal{G}$ and $f,g$ are edges of $\mathcal{G}$.
\begin{itemize}
\item{If $\alpha = \alpha'$, then we can take $\beta = \beta' = \mathrm{id}$.}
\item{If 
\begin{equation*}
\alpha = \mathrm{id} * \mathbf{e}_{f} * \mathrm{id} : x f^{*} f y g^{*} g z \longrightarrow x y g^{*} g z \qquad \text{and} \qquad \alpha' = \mathrm{id} * \mathbf{e}_{g} * \mathrm{id} : x f^{*} f y g^{*} g z \longrightarrow x f^{*} f y z,
\end{equation*}
then we can take
\begin{equation*}
\beta = \mathrm{id} * \mathbf{e}_{g} * \mathrm{id} : x y g^{*} g z \longrightarrow x y z \qquad \text{and} \qquad \beta' = \mathrm{id} * \mathbf{e}_{f} * \mathrm{id} : x f^{*} f y z \longrightarrow x y z.
\end{equation*}}
\item{If 
\begin{equation*}
\alpha = \mathrm{id} * \mathbf{e}_{f} * \mathrm{id} : x f f^{*} f y  \longrightarrow x f y \qquad \text{and} \qquad \alpha' = \mathrm{id} * \mathbf{i}_{f}^{-1} * \mathrm{id} : x f f^{*} f y  \longrightarrow x f y,
\end{equation*}
then we can take $\beta = \beta' = \mathrm{id}$, by Lemma \ref{triidlem}.}
\item{If 
\begin{equation*}
\alpha = \mathrm{id} * \mathbf{e}_{f} * \mathrm{id} : x f^{*} f f^{*} y  \longrightarrow x f^{*} y \qquad \text{and} \qquad \alpha' = \mathrm{id} * \mathbf{i}_{f}^{-1} * \mathrm{id} : x f^{*} f f^{*} y  \longrightarrow x f^{*} y,
\end{equation*}
then we can take $\beta = \beta' = \mathrm{id}$, by Lemma \ref{triidlem}.}
\end{itemize}
All remaining cases are similar to one of the cases above.
\end{proof}

\begin{dfn}
A 1-cell $u$ of $\mathcal{F}_{a} \mathcal{G}$ is \textit{minimal} if there is no simple reduction $u \longrightarrow v$, for any $v$.
\end{dfn}

\begin{lem} \label{diam2}
Let $\alpha: u \longrightarrow v$ and $\alpha': u \longrightarrow v'$ be reductions in $\mathcal{F}_{a} \mathcal{G}$. If $v$ and $v'$ are both minimal, then $v = v'$ and $\alpha = \alpha'$.
\end{lem}

\begin{proof}
We use induction on the length of $u$. If $v = u$ or $v' = u$, then $u$ is minimal and the assertion is true for trivial reasons, so suppose this is not the case. Then we can factor factor $\alpha$ and $\alpha'$ as
\begin{equation*}
u \xrightarrow{\; \alpha_{1} \;} x \xrightarrow{\; \alpha_{2} \;} v \qquad \text{and} \qquad u \xrightarrow{\; \alpha_{1}' \;} x' \xrightarrow{\; \alpha_{2}' \;} v'
\end{equation*}
respectively, where $\alpha_{1}, \alpha_{1}'$ are simple reductions and $\alpha_{2}, \alpha_{2}'$ are reductions. Lemma \ref{diam1} provides us with a commutative square of reductions
\begin{equation*}
\begin{tikzcd}[row sep=huge, column sep=huge]
u \arrow[r, "\alpha_{1}"] \arrow[d, swap, "\alpha_{1}'"] & x \arrow[d, "\beta"] \\
x' \arrow[r, swap, "\beta'"] & y
\end{tikzcd}
\end{equation*}
and we may suppose that $y$ is minimal, by reducing it if necessary. Now $\mathrm{length}(x) < \mathrm{length}(u)$, so $y = v$ and $\beta = \alpha_{2}$ by the induction hypothesis. Applying this same reasoning to $x'$ yields $y = v'$ and $\beta' = \alpha_{2}'$, from which it follows that $v = v'$ and $\alpha = \alpha'$.
\end{proof}

\begin{lem} \label{uniqlem}
Let $u$ be a 1-cell of $\mathcal{F}_{a} \mathcal{G}$ such that $Ru = u$. Then there exists at most one 2-cell $\alpha : u \longrightarrow 1$.
\end{lem}

\begin{proof}
In view of Lemma \ref{diam2}, it suffices to show that every $\alpha : u \longrightarrow 1$ is in fact a reduction. If $\alpha = \mathrm{id}$, there is nothing to prove, so suppose this is not the case. Since $R 1 = 1$, Lemma \ref{simplem} allows us to write $\alpha$ as a finite composition of simple 2-cells. In other words, as a composition in which every component is either a simple reduction or an inverse thereof. We use induction on the length of this composition. If $\alpha$ is equal to
\begin{equation*}
u \xrightarrow{\; \alpha_{1} \;} v \xrightarrow{\; \alpha_{2} \;} 1,
\end{equation*}
with $\alpha_{1}$ a simple reduction, then we are done, for $\alpha_{2}$ is a reduction by the induction hypothesis. If instead $\alpha_{1}^{-1}$ is a simple reduction, let $\beta : u \longrightarrow w$ be any reduction with $w$ minimal. Then $w = 1$ and $\beta \circ \alpha_{1}^{-1} = \alpha_{2}$ by Lemma \ref{diam2}, so $\alpha = \beta$ and we are done as well.
\end{proof}

\begin{thm} \label{coth3}
If $u, v : A \longrightarrow B$ are 1-cells of $\mathcal{F}_{a} \mathcal{G}$, then there exists at most one 2-cell $u \longrightarrow v$.
\end{thm}

\begin{proof}
By Lemma \ref{eqlem}, $v^{*}$ induces a bijection between the set of 2-cells $u \longrightarrow v$ and the set of 2-cells $v^{*} * u \longrightarrow 1$, so we may assume that $v = 1$. Since $R$ is a biequivalence, there is a bijection between the set of 2-cells $u \longrightarrow 1$ and the set of 2-cells $Ru \longrightarrow 1$. By idempotency of $R$, we are now reduced to a situation where the conditions of Lemma \ref{uniqlem} are satisfied.
\end{proof}

\section{Coherence for bigroupoids}

We will now combine the coherence theorem for AU-bigroupoids and the coherence theorem for bicategories into a coherence theorem for bigroupoids using techniques from \cite{MR1250465} and \cite{MR3076451}. Recall that one of the equivalent ways the coherence theorem for bicategories can be expressed is the following.

\begin{thm} \label{catthm}
In a bicategory $\mathcal{B}$, every formal diagram commutes.
\end{thm}

The notion of a formal diagram in a bicategory can be made precise inductively or analogous to Definition \ref{fordiagthm}, but we will not further address this here. Instead, we assume that the reader is familiar with Theorem \ref{catthm} through other sources. A concise proof is given in \cite{1998math.....10017L} for example. In the upcoming Lemma, we shall apply it to partially strictify arbitrary bigroupoids.

\begin{lem} \label{austr}
Given a bigroupoid $\mathcal{B}$, there exists a AU-bigroupoid $\mathcal{SB}$ with biequivalences $(E, \epsilon) : \mathcal{SB} \longrightarrow \mathcal{B}$ and $(S, \sigma) : \mathcal{B} \longrightarrow \mathcal{SB}$.
\end{lem}

\begin{proof}
We start by constructing $\mathcal{SB}$, along with $(E, \epsilon) : \mathcal{SB} \longrightarrow \mathcal{B}$.

The 0-cells of $\mathcal{SB}$ are the same as those of $\mathcal{B}$. The 1-cells of $\mathcal{SB}$ are generated as follows:
\begin{itemize}
\item{If $f$ is a 1-cell of $\mathcal{B}$, then the string $f$ is a 1-cell of $\mathcal{SB}$. For every 0-cell $A$, there is an empty string $\langle \rangle_{A}$ associated to it.}
\item{If $u$ and $v$ are 1-cells of $\mathcal{SB}$ with suitable source and target, then their concatenation $v u$ is also a 1-cell.}
\item{If $u$ is a 1-cell, then its formal inverse $\overline{u}$ is a 1-cell as well.}
\end{itemize}
Composing 1-cells in $\mathcal{SB}$ is done by concatenating. The empty strings serve as identities. The operation $-^{*}$ is given on 1-cells by taking formal inverses.

Before we can finish the definition of $\mathcal{SB}$, we need to define part of $(E, \epsilon)$. On 0-cells, $E$ is the identity. On 1-cells, $E$ evaluates the string, associating to the left and taking formal inverses to (weak) inverses. For example
\begin{equation*}
E(\overline{k \overline{h} g} f ) = (( k * h^{*} ) * g )^{*} * f.
\end{equation*}
 For 1-cells
\begin{equation*}
A \overset{u}{\longrightarrow} B \overset{v}{\longrightarrow} C,
\end{equation*}
of $\mathcal{SB}$, the 2-cell
\begin{equation*}
\epsilon : Ev * Eu \longrightarrow E(v * u)
\end{equation*}
is defined to be the canonical one. The 2-cells
\begin{equation*}
\epsilon : 1_{EA} \longrightarrow E 1_{A} \qquad \text{and} \qquad \epsilon : (Eu)^{*}  \longrightarrow E(u^{*})
\end{equation*}
are both identities.

The set of 2-cells $u \longrightarrow v$ in $\mathcal{SB}$ is defined to be a copy of the set of 2-cells $Eu \longrightarrow Ev$ in $\mathcal{B}$. The vertical composition of 2-cells is borrowed from $\mathcal{B}$ as well. On 2-cells, $E$ is just the identity. In order to define a 2-cell $\alpha$ of $\mathcal{SB}$, it therefore suffices to provide $E \alpha$.

To define the horizontal composition of 2-cells, let $u, u': A \longrightarrow B$ and $v, v' : B \longrightarrow C$ be 1-cells and let $\alpha : u \longrightarrow u'$ and $\beta: v \longrightarrow v'$ be 2-cells of $\mathcal{SB}$. The composition $\beta * \alpha$ is given by requiring that the square
\begin{equation*}
\begin{tikzcd}[row sep=huge, column sep=huge]
Ev * Eu \arrow[r, "\epsilon"] \arrow[d, swap, "E \beta * E \alpha"] & E(v * u) \arrow[d, "E(\beta * \alpha)"] \\
Ev'* Eu' \arrow[r, swap, "\epsilon"] & E(v'* u')
\end{tikzcd}
\end{equation*}
commutes. The operation $-^{*}$ on 2-cells in $\mathcal{B}$ is defined analogously, which boils down to $E(\alpha^{*}) = (E \alpha)^{*}$, as $\epsilon = \mathrm{id}$ in this case. Clearly both $-*-$ and $-^{*}$ are functors. The 2-cell
\begin{equation*}
\mathbf{e}_{u} : u^{*} * u \longrightarrow 1
\end{equation*}
of $\mathcal{SB}$ is defined by
\begin{equation*}
E \mathbf{e}_{u} = \mathbf{e}_{Eu} : Eu^{*} * Eu \longrightarrow 1.
\end{equation*}
Similarly, $\mathbf{i}_{u}$ is represented by $\mathbf{i}_{Eu}$ in $\mathcal{B}$.

Theorem \ref{catthm} can be used to verify that $\mathcal{SB}$ is associative and unital and that the diagrams (\ref{coh4}) and (\ref{coh5}) commute for $\epsilon$. Clearly $E$ is surjective on 0-cells, locally surjective on objects and locally fully faithful.

The morphism $(S, \sigma)$ is the identity on 0-cells, sends a 1-cell to the string with this 1-cell as only element, and is the identity on 2-cells as well. For the composition of 1-cells
\begin{equation*}
A \overset{f}{\longrightarrow} B \overset{g}{\longrightarrow} C,
\end{equation*}
the 2-cell
\begin{equation*}
\sigma : Sg * Sf \longrightarrow S(g * f)
\end{equation*}
is defined by
\begin{equation*}
E \sigma = \mathrm{id} : E(Sg * Sf) \longrightarrow ES(g * f).
\end{equation*}
For identities and inverses, $\sigma$ is defined in a similar way. It is not difficult to check that $(S, \sigma)$ is a morphism. Clearly, $S$ is surjective on 0-cells and locally fully faithful. Lastly, it is locally essentially surjective since a 1-cell $u$ of $\mathcal{SB}$ is isomorphic to $SEu$.
\end{proof}

\begin{dfn}
Let $(F, \phi), (G, \gamma) : \mathcal{A} \longrightarrow \mathcal{B}$ be morphisms of bigroupoids. Assume that $F$ and $G$ agree on 0-cells. Then an \textit{icon} $\alpha : F \Longrightarrow G$ consists of natural isomorphisms
\begin{equation*}
\alpha_{A,B} : F_{A,B} \Longrightarrow G_{A, B},
\end{equation*}
for every pair of 0-cells $A, B$ of $\mathcal{B}$. Furthermore, for every combination
\begin{equation*}
A \overset{f}{\longrightarrow} B \overset{g}{\longrightarrow} C
\end{equation*}
of composable 1-cells of $\mathcal{A}$, the following diagrams should commute
\begin{equation} \label{iconax}
\begin{tikzcd}[row sep=huge, column sep=huge]
Fg * Ff \arrow[r, "\phi"] \arrow[d, swap, "\alpha * \alpha"] & F(g*f) \arrow[d, "\alpha"] & 1_{FA} \arrow[r, "\phi"] \arrow[d, swap, "\mathrm{id}"] & F 1_{A} \arrow[d, "\alpha"] & (Ff)^{*} \arrow[r, "\phi"] \arrow[d, swap, "\alpha^{*}"] & Ff^{*} \arrow[d, "\alpha"] \\
Gg * Gf \arrow[r, swap, "\gamma"] & G(g*f) & 1_{GA} \arrow[r, swap, "\gamma"] & G 1_{A} & (Gf)^{*} \arrow[r, swap, "\gamma"] & Gf^{*}
\end{tikzcd}
\end{equation}
Note that icons may be composed vertically and horizontally, by pointwise composition of the natural isomorphisms.
\end{dfn}

\begin{lem} \label{iconfaith}
Let $(F, \phi), (G, \gamma): \mathcal{A} \longrightarrow \mathcal{B}$ be morphisms of bigroupoids and let $\alpha : F \Longrightarrow G$ be an icon. Then $F$ is locally faithful (locally full) if and only if $G$ is locally faithful (locally full).
\end{lem}

\begin{proof}
This follows from the fact that for every pair of 0-cells $A, B$ of $\mathcal{A}$, the functors $F_{A, B}$ and $G_{A,B}$ are naturally isomorpic by $\alpha_{A, B} : F_{A, B} \Longrightarrow G_{A, B}$.
\end{proof}

We construct a bigroupoid that will act as a (weak) equalizer.

\begin{cns}
Let $(F, \phi), (G, \gamma) : \mathcal{A} \longrightarrow \mathcal{B}$ be morphisms of bigroupoids. We construct a bigroupoid $\mathrm{Eq}(F, G)$ with a strict morphism $P : \mathrm{Eq}(F, G) \longrightarrow \mathcal{A}$ and an icon $\sigma : FP \Longrightarrow GP$.

The 0-cells of $\mathrm{Eq}(F, G)$ are those 0-cells $A \in \mathcal{A}_{0}$ satisfying $FA = GA$. The objects of the groupoid $\mathrm{Eq}(F, G)(A, B)$ are pairs $(f, \alpha)$, with $f : A \longrightarrow B$ a 1-cell in $\mathcal{A}$ and $\alpha : Ff \longrightarrow Gf$ a 2-cell in $\mathcal{B}$. A 2-cell from $(f, \alpha)$ to $(g, \beta)$ is a 2-cell $\delta : f \longrightarrow g$ in $\mathcal{A}$ such that the diagram
\begin{equation} \label{signat}
\begin{tikzcd}[row sep=huge, column sep=huge]
Ff \arrow[r, "\alpha"] \arrow[d, swap, "F \delta"] & Gf \arrow[d, "G \delta"] \\
Fg \arrow[r, swap, "\beta"] & Gg
\end{tikzcd}
\end{equation}
commutes.

Given two 1-cells $(f, \alpha) : A \longrightarrow B$ and $(g, \beta) : B \longrightarrow C$, we define composition by
\begin{equation*}
(g, \beta) * (f, \alpha) = (g*f, \gamma \circ (\beta * \alpha) \circ \phi^{-1}),
\end{equation*}
identity by
\begin{equation*}
1_{A} = (1_{A}, \gamma \circ \phi^{-1})
\end{equation*}
and inverses by
\begin{equation*}
(f, \alpha)^{*} = (f^{*}, \gamma \circ \alpha^{*} \circ \phi^{-1}).
\end{equation*}
On 2-cells of $\mathrm{Eq}(F, G)$, the operations $-*-$ and $-^{*}$ are inherited from $\mathcal{A}$ and  we leave it to the reader to check that the 2-cells of $\mathrm{Eq}(F, G)$ are closed under these operations.

The isomorphisms $\mathbf{a}$, $\mathbf{r}$, $\mathbf{l}$, $\mathbf{e}$ and $\mathbf{i}$ are the same as those of $\mathcal{A}$. We also ask the reader to verify that these are in fact 2-cells of $\mathrm{Eq}(F, G)$, using (\ref{coh4}) and (\ref{coh5}). The fact that the diagrams (\ref{coh1}), (\ref{coh2}) and (\ref{coh3}) commute in $\mathrm{Eq}(F, G)$ follows directly from the fact that they commute in $\mathcal{A}$.

We define the morphism $P : \mathrm{Eq}(F, G) \longrightarrow \mathcal{A}$ by
\begin{equation*}
PA = A, \qquad P_{A, B} (f, \alpha) = f, \qquad P_{A, B} \delta = \delta.
\end{equation*}
It should be clear that is a strict morphism of bigroupoids.

We define the component of the icon $\sigma : FP \Longrightarrow GP$ at a 1-cell $(f, \alpha) : A \longrightarrow B$ by
\begin{equation*}
(\sigma_{A, B})_{(f, \alpha)} = \alpha : Ff \longrightarrow Gf.
\end{equation*}
The naturality of $\sigma_{A, B}$ is immediate by (\ref{signat}). The icon axioms (\ref{iconax}) follow directly from the definition of composition, identity and inversion of 1-cells in $\mathrm{Eq}(F, G)$.
\end{cns}

\begin{lem}
Given a graph $\mathcal{G}$, the free bigroupoid $\mathcal{F}_{b} \mathcal{G}$ on $\mathcal{G}$ exists. We record its universal property:
\begin{itemize}
\item{There exists an inclusion of graphs (the unit of the adjunction), $I_{b} : \mathcal{G} \longrightarrow \mathcal{F}_{b} \mathcal{G}$, such that:}
\item{Given a bigroupoid $\mathcal{B}$ and a morphism $F : \mathcal{G} \longrightarrow \mathcal{B}$ of graphs, there exists a unique strict morphism of bigroupoids $\widetilde{F} : \mathcal{F}_{b} \mathcal{G} \longrightarrow \mathcal{B}$ such that $F = \widetilde{F}I_{b}$.}
\end{itemize}
\end{lem}

\begin{cns}
The construction of $\mathcal{F}_{b} \mathcal{G}$ is analogous to Construction \ref{fracns}.
\end{cns}

\begin{lem} \label{strictify}
Let $F : \mathcal{F}_{b} \mathcal{G} \longrightarrow \mathcal{B}$ be a morphism out of a free bigroupoid. Then there exists a strict morphism $G : \mathcal{F}_{b} \mathcal{G} \longrightarrow \mathcal{B}$ and an icon $\alpha : F \Longrightarrow G$. Furthermore, $FI_{b} = GI_{b} : \mathcal{G} \longrightarrow \mathcal{B}$ and $\alpha I_{b} = \mathrm{id}$ (as $\mathcal{G}_{0} \times \mathcal{G}_{0}$-indexed families of isomorphisms).
\end{lem}

\begin{proof}
By freeness of $\mathcal{F}_{b} \mathcal{G}$, there exists a unique strict morphism $G (= \widetilde{FI_{b}}): \mathcal{F}_{b} \mathcal{G} \longrightarrow \mathcal{B}$ such that $FI_{b} = GI_{b} : \mathcal{G} \longrightarrow \mathcal{B}$. (These and the other morphisms are drawn in the diagram at the bottom of this proof.) The map $I_{b}$ now factors through $P : \mathrm{Eq}(F, G) \longrightarrow \mathcal{F}_{b} \mathcal{G}$ as $P K$, where $K : \mathcal{G} \longrightarrow \mathrm{Eq}(F, G)$
\begin{itemize}
\item{sends a 0-cell $A$ to $A$},
\item{sends a 1-cell $f$ to $(f, \mathrm{id}_{Ff})$}
\item{and sends a 2-cell $\beta$ to $\beta$.}
\end{itemize}
The universal property of $\mathcal{F}_{b} \mathcal{G}$ applied to $K$, gives rise to unique strict morphism $\widetilde{K} : \mathcal{F}_{b} \mathcal{G} \longrightarrow \mathrm{Eq}(F, G)$ satisfying $\widetilde{K} I_{b} = K$. Since $P \widetilde{K} I_{b} = I_{b}$ and $P \widetilde{K}$ is strict, $P \widetilde{K}$ must be the identity, again by the universal property of $\mathcal{F}_{b} \mathcal{G}$. Recall that we have an icon $\sigma : FP \Longrightarrow GP$. The icon $\sigma \widetilde{K}$ therefore has source $F P \widetilde{K} = F$ and target $G P \widetilde{K} = G$, so take $\alpha = \sigma \widetilde{K}$. One easily verifies directly from the definitions of $K$ and $\sigma$  that $\sigma K = \mathrm{id}$. We find
\begin{equation*}
\alpha I_{b} = \sigma \widetilde{K} I_{b} = \sigma K = \mathrm{id},
\end{equation*}
as desired.
\begin{equation*}
\begin{tikzcd}[row sep=huge, column sep=huge]
\mathrm{Eq}(F, G) \arrow[r, shift left, "P"] & \mathcal{F}_{b} \mathcal{G} \arrow[r, shift left, "F"{name=F}] \arrow[r, swap, shift right, "G"{name=G}] \arrow[l, shift left, "\widetilde{K}"] & \mathcal{B} \\
G \arrow[u, "K"] \arrow[ur, swap, "I_{b}"] & &
\end{tikzcd}
\end{equation*} 
\end{proof}

\begin{lem}
Given a graph $\mathcal{G}$, the free 2-groupoid $\mathcal{F}_{s} \mathcal{G}$ on $\mathcal{G}$ exists. We record its universal property:
\begin{itemize}
\item{There exists an inclusion of graphs (the unit of the adjunction), $I_{s} : \mathcal{G} \longrightarrow \mathcal{F}_{s} \mathcal{G}$, such that:}
\item{Given a 2-groupoid $\mathcal{B}$ and a morphism $F : \mathcal{G} \longrightarrow \mathcal{B}$ of graphs, there exists a unique strict morphism of bigroupoids $\widetilde{F} : \mathcal{F}_{s} \mathcal{G} \longrightarrow \mathcal{B}$ such that $F = \widetilde{F} I_{s}$.}
\end{itemize}
\end{lem}

\begin{cns} \label{fscns}
The construction of $\mathcal{F}_{s} \mathcal{G}$ is analogous to Construction \ref{fracns}.
\end{cns}

\begin{thm} \label{coth1}
For every graph $\mathcal{G}$, the strict morphism $\Gamma : \mathcal{F}_{b} \mathcal{G} \longrightarrow \mathcal{F}_{s} \mathcal{G}$, induced by the universal property of $\mathcal{F}_{b} \mathcal{G}$ in the diagram
\begin{equation*}
\begin{tikzcd}[row sep=huge, column sep=huge]
\mathcal{G} \arrow[d, swap, "I_{b}"] \arrow[dr, "I_{s}"] & \\
\mathcal{F}_{b} \mathcal{G} \arrow[r, swap, dashed, "\Gamma"] & \mathcal{F}_{s} \mathcal{G}
\end{tikzcd}
\end{equation*}
is a biequivalence.
\end{thm}

\begin{proof}
It is clear that $\Gamma$ is surjective on 0-cells, since $\mathcal{FG}$ and $\mathcal{F}_{s} \mathcal{G}$ share the same 0-cells (those of $\mathcal{G}$). The fact that $\Gamma$ is locally surjective follows from the fact that $\Gamma$ is locally surjective on the generating 1-cells (the 1-cells of $\mathcal{G}$ and the new 1-cells of the form $1_{A}$) and an easy induction on $-*-$ and $-^{*}$. The local fullness of $\Gamma$ follows from the local discreteness of $\mathcal{F}_{s} \mathcal{G}$, combined with the observation that if $u$ and $v$ are 1-cells of $\mathcal{F}_{b} \mathcal{G}$ such that $\Gamma u = \Gamma v$, then there must have been a 2-cell $u \longrightarrow v$ in $\mathcal{F}_{b} \mathcal{G}$. (This can be made rigorous by comparing the generation of 2-cells in Construction \ref{fracns} with the generation of the congruence relation for 1-cells in Construction \ref{fscns}.)

It remains to show that $\Gamma$ is locally faithful. Let $\Gamma_{1}$ and $\Gamma_{2}$ be the strict morphisms induced by the universal properties of $\mathcal{F}_{b} \mathcal{G}$ and $\mathcal{F}_{a} \mathcal{G}$ respectively, in the diagrams
\begin{equation*}
\begin{tikzcd}[row sep=huge, column sep=huge]
& \mathcal{G} \arrow[dl, swap, "I_{b}"] \arrow[d, "I_{a}"] & \mathcal{G} \arrow[d, swap, "I_{a}"] \arrow[dr, "I_{s}"] & \\
\mathcal{F}_{b} \mathcal{G} \arrow[r, swap, dashed, "\Gamma_{1}"] & \mathcal{F}_{a} \mathcal{G} &\mathcal{F}_{a} \mathcal{G} \arrow[r, swap, dashed, "\Gamma_{2}"] & \mathcal{F}_{s} \mathcal{G}
\end{tikzcd}
\end{equation*}
Then by uniqueness of $\Gamma$, we obtain the factorization $\Gamma = \Gamma_{2} \Gamma_{1}$. Since $\Gamma_{2}$ is locally faithful as a trivial consequence of Theorem \ref{coth3}, it suffices to show that $\Gamma_{1}$ is locally faithful.

Recall that by Lemma \ref{austr} there is a locally faithful morphism $S : \mathcal{F}_{b} \mathcal{G} \longrightarrow \mathcal{B}$ into a AU-bigroupoid. By Lemma \ref{strictify}, there exists a strict morphism $T : \mathcal{F}_{b} \mathcal{G} \longrightarrow \mathcal{B}$ along with an icon $\alpha : S \Longrightarrow T$. Note that the presence of this icon guarantees that $T$ is locally faithful as well, by virtue of Lemma \ref{iconfaith}. We use the universal property of $\mathcal{F}_{a} \mathcal{G}$ to find a unique strict morphism $T_{a} ( = \widetilde{T I_{b}}) : \mathcal{F}_{a} \mathcal{G} \longrightarrow \mathcal{B}$ satisfying $T_{a} I_{a} =  T I_{b} $. This gives
\begin{equation*}
T_{a} \Gamma_{1} I_{b} = T_{a} I_{a} =  T I_{b},
\end{equation*}
which implies $T_{a} \Gamma_{1} = T$, by the universal property of $\mathcal{F}_{b} \mathcal{G}$. But then $\Gamma_{1}$ must be locally faithful, as $T$ is.
\end{proof}

\begin{dfn} \label{fordiagdef}
Given a bigroupoid $\mathcal{B}$, we can construct the free bigroupoid $\mathcal{F}_{b} \mathcal{B}$ on its underlying graph and consider the obvious strict morphism (the counit of the adjunction), $J_{b} : \mathcal{F}_{b} \mathcal{B} \longrightarrow \mathcal{B}$. A diagram (consisting of 2-cells), in $\mathcal{B}$ is called a \textit{formal diagram} if it is the image of a diagram in $\mathcal{F}_{b} \mathcal{B}$, under $J_{b}$. If such a formal diagram happens to consist of only a single 2-cell, we will call this 2-cell \textit{canonical}.
\end{dfn}

\begin{thm} \label{fordiagthm}
In a bigroupoid $\mathcal{B}$, every formal diagram commutes.
\end{thm}

\begin{proof}
Since $\mathcal{F}_{s} \mathcal{B}$ is locally discrete and $\Gamma : \mathcal{F}_{b} \mathcal{B} \longrightarrow \mathcal{F}_{s} \mathcal{B}$ is locally faithful by Theorem \ref{coth1}, every diagram of 2-cells commutes in $\mathcal{F}_{b} \mathcal{B}$. Trivially, their images under $J_{b}$ commute as well.
\end{proof}

\section{Coherence for morphisms}

In this section we prove a coherence theorem for morphisms of bigroupoids. The proof that we give below is essentially the one given in \cite{MR3076451} for morphisms of bicategories. The approach of \cite{MR3076451} is in turn based on that of \cite{MR1250465}.

\begin{lem}
Given a morphism $F : \mathcal{G} \longrightarrow \mathcal{G}'$ of graphs, the free morphism (of bigroupoids) $\mathcal{F}_{m} F : \mathcal{F}_{b} \mathcal{G} \longrightarrow \mathcal{F}_{m} \mathcal{G}'$ on $F$ exists. We record its universal property:
\begin{itemize}
\item{There exists a commutative square (of graphs)
\begin{equation*}
\begin{tikzcd}[row sep=huge, column sep=huge]
\mathcal{G} \arrow[r, "F"] \arrow[d, swap, "I_{b}"] & \mathcal{G}' \arrow[d, "I_{m}"] \\
\mathcal{F}_{b} \mathcal{G} \arrow[r, swap, "\mathcal{F}_{m}F"] & \mathcal{F}_{m} \mathcal{G}'
\end{tikzcd}
\end{equation*}
such that:}
\item{Given a commutative square (of graphs)
\begin{equation*}
\begin{tikzcd}[row sep=huge, column sep=huge]
\mathcal{G} \arrow[r, "F"] \arrow[d, swap, "R"] & \mathcal{G}' \arrow[d, "S"] \\
\mathcal{A} \arrow[r, swap, "G"] & \mathcal{B}
\end{tikzcd}
\end{equation*}
with $G : \mathcal{A} \longrightarrow \mathcal{B}$ a morphism of bigroupoids, there exists a unique square (of bigroupoids)
\begin{equation*}
\begin{tikzcd}[row sep=huge, column sep=huge]
\mathcal{F}_{b} \mathcal{G} \arrow[r, "\mathcal{F}_{m}F"] \arrow[d, swap, "\widetilde{R}"] & \mathcal{F}_{m} \mathcal{G}' \arrow[d, "\widetilde{S}"] \\
\mathcal{A} \arrow[r, swap, "G"] & \mathcal{B}
\end{tikzcd}
\end{equation*}
such that $R = \widetilde{R} I_{b}$ and $S = \widetilde{S} I_{m}$, with $\widetilde{R}$ and $\widetilde{S}$ strict.}
\end{itemize}
\end{lem}

\begin{cns} \label{cns3}
We sketch the construction of $\mathcal{F}_{m} \mathcal{G}'$, from which it should be clear how $\mathcal{F}_{m} F : \mathcal{F}_{b} \mathcal{G} \longrightarrow \mathcal{F}_{m} \mathcal{G}'$ is defined. We leave it to the reader to fill in the necessary details.

The 0-cells of $\mathcal{F}_{m} \mathcal{G}'$ are the nodes of $\mathcal{G}'$. For every node $A$ of $\mathcal{G}'$, we add a new edge $1_{A} : A \longrightarrow A$ and for every 1-cell $f : B \longrightarrow C$ of $\mathcal{FG}$, we add a new edge $\mathcal{F}_{m}Ff : FB \longrightarrow FC$. We formally close the edges under the operations $- * -$ and $-^{*}$, taking into account the sources and targets in the obvious way. We quotient out by the congruence relation generated by the requirement that if edges $f$ of $\mathcal{G}$ and $g$ of $\mathcal{G}'$ satisfy $Ff = g$, then $\mathcal{F}_{m}Ff \sim g$. The 1-cells of $\mathcal{F}_{m} \mathcal{G}'$ are the equivalence classes under this quotient.

For 1-cells 
\begin{equation*}
A \overset{f}{\longrightarrow} B \overset{g}{\longrightarrow} C \overset{h}{\longrightarrow} D
\end{equation*}
of $\mathcal{F}_{m} \mathcal{G}'$, we create 2-cells $\mathbf{a}_{h, g, f}$, $\mathbf{l}_{f}$, $\mathbf{r}_{f}$, $\mathbf{e}_{f}$, $\mathbf{i}_{f}$, $\mathbf{a}_{h, g, f}^{-1}$, $\mathbf{l}_{f}^{-1}$, $\mathbf{r}_{f}^{-1}$, $\mathbf{e}_{f}^{-1}$, $\mathbf{i}_{f}^{-1}$ and $\mathrm{id}_{f}$. For 1-cells 
\begin{equation*}
A \overset{f}{\longrightarrow} B \overset{g}{\longrightarrow} C
\end{equation*}
of $\mathcal{FG}$, we add 2-cells $\phi_{g,f}$, $\phi_{A}$, $\phi_{f}$, $\phi_{g,f}^{-1}$, $\phi_{A}^{-1}$ and $\phi_{f}^{-1}$. We close the 2-cells under the operations $- * -$, $-^{*}$ and $- \circ -$ (whenever these operations make sense). We quotient out by the congruence relation generated by the requirements that $- \circ -$ is associative; $\mathrm{id}$ acts as identity; $- ^{-1}$ acts as inverse; $- * -$ and $-^{*}$ are functors; $\mathbf{a}$, $\mathbf{l}$, $\mathbf{r}$, $\mathbf{e}$, $\mathbf{i}$ and $\phi$ are natural; the coherence laws (\ref{coh1}), (\ref{coh2}), (\ref{coh3}), (\ref{coh4}) and (\ref{coh5}) hold; and $\mathcal{F}_{m} F$ is locally a functor. The 2-cells of $\mathcal{F}_{m} \mathcal{G}'$ are the equivalence classes under this quotient.
\end{cns}

\begin{lem} \label{2sqr}
Consider, for $i = 1,2$, the commuting squares (of graphs)
\begin{equation} \label{sifisq}
\begin{tikzcd}[row sep=huge, column sep=huge]
\mathcal{G} \arrow[r, "G"] \arrow[d, swap, "R"] & \mathcal{G}' \arrow[d, "S"] \\
\mathcal{A} \arrow[r, swap, "F_{i}"] & \mathcal{B}
\end{tikzcd}
\end{equation}
with $(F_{i}, \phi_{i}) : \mathcal{A} \longrightarrow \mathcal{B}$ morphisms of bigroupoids. Let
\begin{equation*}
\begin{tikzcd}[row sep=huge, column sep=huge]
\mathcal{F}_{b} \mathcal{G} \arrow[r, "\mathcal{F}_{m} G"] \arrow[d, swap, "\widetilde{R}"] & \mathcal{F}_{m} \mathcal{G}' \arrow[d, "\widetilde{S}_{i}"] \\
\mathcal{A} \arrow[r, swap, "F_{i}"] & \mathcal{B}
\end{tikzcd}
\end{equation*}
be the squares induced by the universal property of $\mathcal{F}_{m} G$. (Note that in general the $\widetilde{S}_{i}$ are distinct, since they depend on the $F_{i}$.) Assume that $F_{1}$ and $F_{2}$ agree on 0-cells. Then if $\alpha : F_{1} \Longrightarrow F_{2}$ is an icon such that
\begin{equation} \label{arbg}
\alpha R = \mathrm{id}
\end{equation}
as $\mathcal{G}_{0} \times \mathcal{G}_{0}$-indexed families of isomorphisms, there is an icon $\beta : \widetilde{S}_{1} \Longrightarrow \widetilde{S}_{2}$ such that
\begin{equation*}
\alpha \widetilde{R} = \beta \mathcal{F}_{m} G
\end{equation*}
as icons.
\end{lem}

\begin{proof}
We construct a new bigroupoid $\mathcal{B}^{I}$ out of $\mathcal{B}$. The 0-cells of $\mathcal{B}^{I}$ are the same as those of $\mathcal{B}$. A 1-cell in $\mathcal{B}^{I}$, from $A$ to $B$, is a 2-cell $\gamma : g_{1} \longrightarrow g_{2}$ in $\mathcal{B}$ with $g_{1}, g_{2} : A \longrightarrow B$. For convenience, we make the domain and codomain explicit in our notation $(g_{1}, g_{2}, \gamma)$ for such a 1-cell. A 2-cell in $\mathcal{B}^{I}$, from $(g_{1}, g_{2}, \gamma)$ to $(h_{1}, h_{2}, \delta)$, is a pair $(\sigma_{1}, \sigma_{2})$ of 2-cells in $\mathcal{B}$ such that the square
\begin{equation*}
\begin{tikzcd}[row sep=huge, column sep=huge]
g_{1} \arrow[r, "\sigma_{1}"] \arrow[d, swap, "\gamma"] & h_{1} \arrow[d, "\delta"] \\
g_{2} \arrow[r, swap, "\sigma_{2}"] & h_{2}
\end{tikzcd}
\end{equation*}
commutes. Composition of 2-cells is done pointwise.

The identity 1-cell on a 0-cell $A$ is given by $\mathrm{id}_{1_{A}}$. The operations $- * -$ and $-^{*}$ on 1-cells of $\mathcal{B}^{I}$ are given by these same operations in $\mathcal{B}$ (but as 2-cells there). The operations $- * -$ and $-^{*}$ on 2-cells of $\mathcal{B}^{I}$ are also the same as in $\mathcal{B}$ (pointwise). The 2-cells $\mathbf{a}$ are taken from $\mathcal{B}$, as in the commuting square
\begin{equation*}
\begin{tikzcd}[row sep=huge, column sep=huge]
(k_{1}  h_{1})  g_{1} \arrow[r, "\mathbf{a}"] \arrow[d, swap, "(\epsilon * \delta) * \gamma"] & k_{1}  (h_{1}  g_{1}) \arrow[d, "\epsilon * ( \delta * \gamma)"] \\
(k_{2}  h_{2})  g_{2} \arrow[r, swap, "\mathbf{a}"] & k_{2}  (h_{2}  g_{2})
\end{tikzcd}
\end{equation*}
Similar commuting squares exist for $\mathbf{l}$, $\mathbf{r}$, $\mathbf{e}$ and $\mathbf{i}$. Commutativity of (\ref{coh1}), (\ref{coh2}) and (\ref{coh3}) in $\mathcal{B}^{I}$ follows directly from their commutativity in $\mathcal{B}$.

Note that there are two strict morphisms of bigroupoids $P_{i} : \mathcal{B}^{I} \longrightarrow \mathcal{B}$, for $i = 1, 2$, which
\begin{itemize}
\item{send a 0-cell $A$ to $A$,
\item{send a 1-cell $(g_{1}, g_{2}, \gamma)$ to $g_{i}$}
\item{and send a 2-cell $(\sigma_{1}, \sigma_{2})$ to $\sigma_{i}$,}}
\end{itemize}
together with an icon $\pi : P_{1} \Longrightarrow P_{2}$, whose component at a 1-cell $(g_{1}, g_{2}, \gamma) : A \longrightarrow B$ is given by
\begin{equation*}
(\pi_{A, B})_{(g_{1}, g_{2}, \gamma)} = \gamma : g_{1} \longrightarrow g_{2}.
\end{equation*}
The icon axioms (\ref{iconax}) are easily seen to hold.

The icon $\alpha : F_{1} \Longrightarrow F_{2}$ induces a morphism of bigroupoids $(F, \phi) : \mathcal{A} \longrightarrow \mathcal{B}^{I}$, which
\begin{itemize}
\item{sends a 0-cell $A$ to $F_{1}A$ (which is the same as $F_{2}A)$},
\item{sends a 1-cell $f : A \longrightarrow B$ to $(\alpha_{A,B})_{f}$},
\item{sends a 2-cell $\sigma$ to $(F_{1} \sigma, F_{2} \sigma)$}
\item{and has $\phi = (\phi_{1}, \phi_{2})$.}
\end{itemize}
The fact that the $\phi$ are legitimate 2-cells follows from the icon axioms (\ref{iconax}). Commutativity of (\ref{coh4}) and (\ref{coh5}) for $\phi$ follows from the fact that these diagrams commute for $\phi_{1}$ and $\phi_{2}$.

There is also an obvious morphisms of graphs $T : \mathcal{G}' \longrightarrow \mathcal{B}^{I}$, induced by $S$. This gives a square, which commutes by (\ref{sifisq}) and (\ref{arbg}) and which produces a second square
\begin{equation*}
\begin{tikzcd}[row sep=huge, column sep=huge]
\mathcal{G} \arrow[r, "G"] \arrow[d, swap, "R"] & \mathcal{G}' \arrow[d, "T"] \\
\mathcal{A} \arrow[r, swap, "F"] & \mathcal{B}^{I}
\end{tikzcd}
\qquad
\begin{tikzcd}
{} \arrow[r, squiggly] & {}
\end{tikzcd}
\qquad
\begin{tikzcd}[row sep=huge, column sep=huge]
\mathcal{F}_{b} \mathcal{G} \arrow[r, "\mathcal{F}_{m} G"] \arrow[d, swap, "\widetilde{R}"] & \mathcal{F}_{m} \mathcal{G}' \arrow[d, "\widetilde{T}"] \\
\mathcal{A} \arrow[r, swap, "F"] & \mathcal{B}^{I}
\end{tikzcd}
\end{equation*}
via the universal property of $\mathcal{F}_{m} G$. It is clear that $P_{i} F = F_{i}$, so
\begin{equation*}
P_{i} \widetilde{T} \mathcal{F}_{m} G = P_{i} F \widetilde{R} = F_{i} \widetilde{R},
\end{equation*}
which implies that $P_{i} \widetilde{T} = \widetilde{S}_{i}$ by the universal property of $\mathcal{F}_{m} G$. This allows us to define
\begin{equation*}
\beta = \pi \widetilde{T} : \widetilde{S}_{1} \Longrightarrow \widetilde{S}_{2}.
\end{equation*}
One easily verifies that $\pi F = \alpha$, by definition of $\pi$ and $F$, which shows that
\begin{equation*}
\beta \mathcal{F}_{m} G = \pi \widetilde{T} \mathcal{F}_{m} G = \pi F \widetilde{R} = \alpha \widetilde{R},
\end{equation*}
as needed.
\end{proof}

\begin{thm} \label{coth2}
For every morphism of graphs $F : \mathcal{G} \longrightarrow \mathcal{G}'$, the strict morphism $\Delta: \mathcal{F}_{m} \mathcal{G}' \longrightarrow \mathcal{F}_{s} \mathcal{G} '$ induced by the universal property of $\mathcal{F}_{m} F$ in the diagram
\begin{equation*}
\begin{tikzcd}[row sep=huge, column sep=huge]
\mathcal{G} \arrow[r, "F"] \arrow[d, swap, "I_{b}"] \arrow[dd, swap, bend right=50, "I_{s}"] & \mathcal{G}' \arrow[d, "I_{m}"] \arrow[dd, bend left=50, "I_{s}'"] \\
\mathcal{F}_{b} \mathcal{G} \arrow[r, "\mathcal{F}_{m} F"] \arrow[d, swap, dashed, "\Gamma"] & \mathcal{F}_{m} \mathcal{G}' \arrow[d, dashed, "\Delta"] \\
\mathcal{F}_{s} \mathcal{G} \arrow[r, swap, "\mathcal{F}_{s} F"] & \mathcal{F}_{s} \mathcal{G}'
\end{tikzcd}
\end{equation*}
is a biequivalence.
\end{thm}

\begin{proof}
Surjectivity on 0-cells, local surjectivity and local fullness for $\Delta$ can be proven in the same way as was done for $\Gamma$ in the proof of Theorem \ref{coth1}. All that is left to show is that $\Delta$ is locally faithful.

By Lemma \ref{strictify}, there exists a strict morphism $S : \mathcal{F}_{b} \mathcal{G} \longrightarrow \mathcal{F}_{m} \mathcal{G}'$ along with an icon $\alpha : \mathcal{F}_{m} F \Longrightarrow S$, such that $S \circ I_{b} = \mathcal{F}_{m} F \circ I_{b}$ and $\alpha I_{b} = \mathrm{id}$. Since $S \circ I_{b} = \mathcal{F}_{m} F \circ I_{b}$, we have two commuting squares
\begin{equation*}
\begin{tikzcd}[row sep=huge, column sep=huge]
\mathcal{G} \arrow[r, "F"] \arrow[d, swap, "I_{b}"] & \mathcal{G}' \arrow[d, "I_{m}"] & \mathcal{G} \arrow[r, "F"] \arrow[d, swap, "I_{b}"] & \mathcal{G}' \arrow[d, "I_{m}"] \\
\mathcal{F}_{b} \mathcal{G} \arrow[r, swap, "S"] & \mathcal{F}_{m} \mathcal{G}' & \mathcal{F}_{b} \mathcal{G} \arrow[r, swap, "\mathcal{F}_{m} F"] & \mathcal{F}_{m} \mathcal{G}'
\end{tikzcd}
\end{equation*}
The equality $\alpha I = \mathrm{id}$ shows that we may apply Lemma \ref{2sqr} to find an icon $\beta : \mathrm{id} \Longrightarrow E$, where $E$ is produced by the universal property of $\mathcal{F}_{m} F$ via
\begin{equation*}
\begin{tikzcd}[row sep=huge, column sep=huge]
\mathcal{G} \arrow[r, "F"] \arrow[d, swap, "I_{b}"] & \mathcal{G}' \arrow[d, "I_{m}"] \\
\mathcal{F}_{b} \mathcal{G} \arrow[r, swap, "S"] & \mathcal{F}_{m} \mathcal{G}'
\end{tikzcd}
\qquad
\begin{tikzcd}
{} \arrow[r, squiggly] & {}
\end{tikzcd}
\qquad
\begin{tikzcd}[row sep=huge, column sep=huge]
\mathcal{F}_{b} \mathcal{G} \arrow[r, "\mathcal{F}_{m} F"] \arrow[d, swap, "\mathrm{id}"] & \mathcal{F}_{m} \mathcal{G}' \arrow[d, "E"] \\
\mathcal{F}_{b} \mathcal{G} \arrow[r, swap, "S"] & \mathcal{F}_{m} \mathcal{G}'
\end{tikzcd}
\end{equation*}
Since the identity morphism is locally fully faithful, so is $E$ by Lemma \ref{iconfaith}.

Now the universal property of $\mathcal{F}_{m} F$ induces a square
\begin{equation*}
\begin{tikzcd}[row sep=huge, column sep=huge]
\mathcal{G} \arrow[r, "F"] \arrow[d, swap, "I_{b}"] & \mathcal{G}' \arrow[d, "I_{b}'"] \\
\mathcal{F}_{b} \mathcal{G} \arrow[r, swap, "\mathcal{F}_{b} F"] & \mathcal{F}_{b} \mathcal{G}'
\end{tikzcd}
\qquad
\begin{tikzcd}
{} \arrow[r, squiggly] & {}
\end{tikzcd}
\qquad
\begin{tikzcd}[row sep=huge, column sep=huge]
\mathcal{F}_{b} \mathcal{G} \arrow[r, "\mathcal{F}_{m} F"] \arrow[d, swap, "\mathrm{id}"] & \mathcal{F}_{m} \mathcal{G}' \arrow[d, "\Delta_{1}"] \\
\mathcal{F}_{b} \mathcal{G} \arrow[r, swap, "\mathcal{F}_{b} F"] & \mathcal{F}_{b} \mathcal{G}'
\end{tikzcd}
\end{equation*}
Consider $\Gamma' : \mathcal{F}_{b} \mathcal{G}' \longrightarrow \mathcal{F}_{s} \mathcal{G}'$. We claim that $\Gamma' \circ \Delta_{1} = \Delta$. First note that
\begin{equation*}
\Gamma' \circ \mathcal{F}_{b} F \circ I_{b} = \Gamma' \circ I_{b}' \circ F = I_{s}' \circ F = \mathcal{F}_{s} F \circ I_{s} = \mathcal{F}_{s} F \circ \Gamma \circ I_{b},
\end{equation*}
since $I_{b}$ and $I_{b}'$ are components of the unit for $\mathcal{F}_{b}$; by definition of $\Gamma'$; since $I_{s}$ and $I_{s}'$ are components of the unit for $\mathcal{F}_{s}$; and by definition of $\Gamma$. The universal property of $\mathcal{F}_{b} \mathcal{G}$ now dictates that $\Gamma' \circ \mathcal{F}_{b} F = \mathcal{F}_{s} F \circ \Gamma$ and thus
\begin{equation} \label{deltaeq1}
\Gamma' \circ \Delta_{1} \circ \mathcal{F}_{m} F = \Gamma' \circ \mathcal{F}_{b} F = \mathcal{F}_{s} F \circ \Gamma.
\end{equation}
Moreover,
\begin{equation} \label{deltaeq2}
\Gamma' \circ \Delta_{1} \circ I_{m} = \Gamma' \circ I_{b}' = I_{s}'
\end{equation}
by definition of $\Delta_{1}$ and $\Gamma'$. But now equations (\ref{deltaeq1}) and (\ref{deltaeq2}) combined imply $\Gamma' \circ \Delta_{1} = \Delta$, using the universal property of $\mathcal{F}_{m} F$. The upshot of this is that for $\Delta$ to be locally faithful, it suffices that $\Delta_{1}$ is, as $\Gamma'$ is locally faithful by Theorem \ref{coth1}.

Let $\widetilde{I}_{m} : \mathcal{F}_{b} \mathcal{G}' \longrightarrow \mathcal{F}_{m} \mathcal{G}'$ be the unique strict morphism such that $I_{m} = \widetilde{I}_{m} I_{b}'$, given by the universal property of $\mathcal{F}_{b} \mathcal{G}'$. We claim that $E = \widetilde{I}_{m} \circ \Delta_{1}$. This will finish the proof, because we have established that $E$ is locally faithful. Note that
\begin{equation*}
\widetilde{I}_{m} \circ \mathcal{F}_{b} F \circ I_{b} = \widetilde{I}_{m} \circ I_{b}' \circ F = I_{m} \circ F = \mathcal{F}_{m} F \circ I_{b} = S \circ I_{b},
\end{equation*}
since $I_{b}$ and $I_{b}'$ are components of the unit for $\mathcal{F}_{b}$; by definition of $\widetilde{I}_{m}$; by definition of $\mathcal{F}_{m} F$; and by choice of $S$. Hence $\widetilde{I}_{m} \circ \mathcal{F}_{b} F = S$ by the universal property of $\mathcal{F}_{b} \mathcal{G}$ and thus
\begin{equation} \label{keq1}
\widetilde{I}_{m} \circ \Delta_{1} \circ \mathcal{F}_{m} F = \widetilde{I}_{m} \circ \mathcal{F}_{b} F = S.
\end{equation}
Moreover,
\begin{equation} \label{keq2}
\widetilde{I}_{m} \circ \Delta_{1} \circ I_{m} = \widetilde{I}_{m} \circ I_{b}' = I_{m}
\end{equation}
by definition of $\Delta_{1}$ and $\widetilde{I}_{m}$. Equations (\ref{keq1}) and (\ref{keq2}) combined imply $E = \widetilde{I}_{m} \circ \Delta_{1}$, using the universal property of $\mathcal{F}_{m} F$.
\end{proof}

\begin{dfn}
Given a morphism of bigroupoids $(F, \phi) : \mathcal{A} \longrightarrow \mathcal{B}$, we can construct the free morphism $\mathcal{F}_{m} F : \mathcal{F}_{b} \mathcal{A} \longrightarrow \mathcal{F}_{m} \mathcal{B}$ on the underlying morphism of graphs and consider the obvious strict morphism (a component of the counit of the adjunction), $J_{m} : \mathcal{F}_{m} \mathcal{B} \longrightarrow \mathcal{B}$. A diagram (consisting of 2-cells), in $\mathcal{B}$ is called a \textit{formal $\phi$-diagram} if it is the image of a diagram in $\mathcal{F}_{m} \mathcal{B}$, under $J_{m}$.
\end{dfn}

\begin{thm} \label{phifordiagthm}
Let $(F, \phi) : \mathcal{A} \longrightarrow \mathcal{B}$ be a morphism of bigroupoids. Then every formal $\phi$-diagram commutes in $\mathcal{B}$.
\end{thm}

\begin{proof}
Since $\mathcal{F}_{s} \mathcal{B}$ is locally discrete and $\Delta : \mathcal{F}_{m} \mathcal{B} \longrightarrow \mathcal{F}_{s} \mathcal{B}$ is locally faithful by Theorem \ref{coth2}, every diagram of 2-cells commutes in $\mathcal{F}_{m} \mathcal{B}$. Trivially, their images under $J_{b}$ commute as well.
\end{proof}

\begin{rmk}
Theorems \ref{coth1} and \ref{coth2} are formulated in terms of free bigroupoids on a \textit{graph}. It is possible to make an analogous (stronger) statement involving free bigroupoids on a \textit{groupoid enriched graph}. This is similar to what is done in \cite{MR1250465} for monoidal categories and in \cite{MR3076451} for bicategories. We chose the former version, since it is sufficient for our purposes. However, the latter version is valid as well and can be proven without too much extra effort. Here is a rough outline of the proof. Using Theorem \ref{coth3}, one can show that every AU-bigroupoid is biequivalent to a 2-groupoid, using a construction similar to Lemma \ref{austr}. Additionally, Lemma \ref{strictify} is also valid for $F : \mathcal{F}_{a} \mathcal{G} \longrightarrow \mathcal{B}$, with $\mathcal{B}$ an AU-bigroupoid, by the same proof. Using this, one can show that (the new) $\Gamma_{2}$ is locally faithful in the same way as was done for $\Gamma_{1}$ in the proof of Theorem \ref{coth1}. The rest of the structure of the proof stays the same. For the individual Lemmas, it will be useful to refer to \cite{MR3076451} as well, as some details involving 2-cells have been lost due to simplifications we could make by working with graphs instead of groupoid enriched graphs.
\end{rmk}

\bibliographystyle{alpha}
\bibliography{Bigroupoidbib}

\end{document}